\documentclass[11pt]{amsart}
\usepackage{amsmath,amsthm, amscd, amssymb, amsfonts, mathtools}
\usepackage{placeins}
\usepackage[all]{xy}
\usepackage{mathrsfs}
\usepackage[inline]{enumitem}
\usepackage{multicol}
\usepackage[T2A,T1]{fontenc}
\usepackage{rotating}
\usepackage{hyperref} 
\usepackage{diagbox}
\usepackage[ansinew]{inputenc}
\usepackage{graphicx,fancyhdr}
\usepackage{color}
\usepackage[normalem]{ulem}
\usepackage{cancel}
\usepackage{nicefrac,xfrac}
\usepackage{float}

\hypersetup{colorlinks=true,citecolor=cyan,linkcolor=blue,linktocpage=true}

\newcommand{\nucleo}{\mathbf R}

\allowdisplaybreaks

\numberwithin{equation}{section}
\newtheorem{theorem}{Theorem}[section]
\newtheorem{lemma}[theorem]{Lemma}
\newtheorem{coro}[theorem]{Corollary}

\newtheorem{prop}[theorem]{Proposition}

\theoremstyle{definition}
\newtheorem{definition}[theorem]{Definition}

\theoremstyle{remark}
\newtheorem{remark}[theorem]{Remark}

\newcommand{\pf}{\begin{proof}}
	\newcommand{\epf}{\end{proof}}

\newcommand{\ku}{ \Bbbk}

\newcommand{\kut}{ \ku^{\times}}

\newcommand{\G}{\mathbb G}

\newcommand{\ghost}{\mathscr{G}}

\newcommand{\I}{\mathbb I}

\newcommand{\N}{\mathbb N}

\newcommand{\Z}{\mathbb Z}

\newcommand{\lmod}[1]{{}_{ #1 }\mathcal{M}}

\newcommand{\cA}{\mathcal{A}}

\newcommand{\cO}{\mathcal{O}}
\newcommand{\cW}{\mathcal{W}}

\newcommand{\cF}{\mathcal{F}}

\newcommand{\cL}{\mathscr{L}}
\newcommand{\Pc}{{\mathcal P}}

\newcommand{\Ss}{{\mathcal S}}

\newcommand{\cV}{\mathcal{V}}
\newcommand{\cX}{\mathcal{X}}

\newcommand{\DLstr}{\mathcal{DL}}

\newcommand{\g}{\mathfrak g}

\newcommand{\Lstr}{\mathfrak L}

\newcommand{\lstr}{\mathfrak l}

\newcommand{\car}{\operatorname{char}}

\newcommand{\id}{\operatorname{id}}

\newcommand{\GK}{\operatorname{GK-dim}}
\newcommand{\Hom}{\operatorname{Hom}}
\newcommand{\ord}{\operatorname{ord}}

\newcommand{\irr}{\operatorname{irrep}}
\newcommand{\Ind}{\operatorname{Ind}}

\newcommand{\toba}{\mathscr{B}}

\newcommand{\ot}{\otimes}

\newcommand{\yd}[1]{{}^{ #1 }_{ #1 }\mathcal{YD}}

\newcommand{\Gb}{\mathbf G}

\newcommand{\xb}{x}
\newcommand{\ub}{u}

\newcommand{\xrightarrowdbl}[2][]{%
	\xrightarrow[#1]{#2}\mathrel{\mkern-14mu}\rightarrow
}

\newcounter{tabla}\stepcounter{tabla}

\begin{document}
	
	\title[On the Drinfeld doubles of Laistrygonian algebras]{On the Drinfeld doubles of Laistrygonian algebras in positive characteristic}

	\author[Andruskiewitsch]{Nicol\'as Andruskiewitsch}
	\address[N. Andruskiewitsch]{CIEM-CONICET, Profesor Em\'erito, Facultad de Matem\'atica, Astronom\'ia, F\'isica y Computaci\'on, Universidad Nacional de C\'ordoba, Medina Allende s/n, Ciudad Universitaria, 5000 C\'ordoba, Argentina
		\vspace{10pt}\newline
		Department of Mathematics and Data Science, Vrije Universiteit Brussel, Pleinlaan 2, 1050 Brussels, Belgium
		\vspace{10pt}\newline
		Shenzhen International Center for Mathematics, Southern University of Science and Technology, Shenzhen 518055, China}
	
	\email{nicolas.andruskiewitsch@unc.edu.ar}
	
	\author[Bagio]{Dirceu Bagio}
	\address[D. Bagio]{Departamento de Matem\'atica, Universidade Federal de Santa Catarina, 88040-900 Florian\'opolis, SC, Brazil}
	\email{d.bagio@ufsc.br}
	
	\author[Flora]{Saradia Della Flora}
	\address[S. Della Flora]{Departamento de Matem\'atica, Universidade Federal de Santa Maria, 97105-900 Santa Maria, RS, Brazil}
	\email{saradia.flora@ufsm.br}
	
	\author[Flores]{Daiana Flores}
	\address[D. Flores]{Departamento de Matem\'atica, Universidade Federal de Santa Maria, 97105-900 Santa Maria, RS, Brazil}
	\email{flores@ufsm.br}
	
	\author[M\'arquez]{Oscar M\'arquez}
	\address[O. M\'arquez]{Departamento de Matem\'atica, Universidade Federal de Santa Catarina, 88040-900 Florian\'opolis, SC, Brazil}
	\email{oscar.marquez.sosa@ufsc.br}
	
	\author[Pollastri]{H\'ector Mart\'\i n Pe\~na Pollastri}
	\address[H.~M. Pe\~na Pollastri]{Department of Mathematics, Indiana University, Bloomington, IN 47405, USA}
	\email{hpenapol@iu.edu}
	
	\thanks{
		The work of N. Andruskiewitsch was partially supported by 
		CONICET (PIP 11220200102916CO), FONCyT-ANPCyT (PICT-2019-03660), and
		the Secyt (UNC) (Proyecto Consolidar 33620230100517CB). The work of D. Bagio was partially supported by Funda\c c\~ao de Amparo a Pesquisa e Inova\c c\~ao do Estado de Santa Catarina (FAPESC), Edital 21/2024. The work of H. M. P. Pollastri was supported by the Simons Foundation Award 889000 as part of the Simons Collaboration on Global Categorical Symmetries and by the National Science Foundation under Grant No. 2529637.}

\begin{abstract} 
In this work, we study restricted Laistrygonian algebras over a field of positive characteristic $p$. We compute their Drinfeld double $\mathtt{D_p(q,a)}$ and give explicit presentations by generators and relations. We prove that $\mathtt{D_p(q,a)}$ is a pointed Hopf algebra. For $\mathtt{D_p(1,a)}$, we describe an abelian exact sequence of Hopf algebras and identify one of its quotients as a restricted enveloping Lie algebra. This exact sequence enables us to classify the finite-dimensional simple $\mathtt{D_p(1,a)}$-modules. We also introduce an infinite-dimensional Hopf algebra covering the restricted Drinfeld double and study its relation with the corresponding finite-dimensional case. Using this covering together with a Frobenius map, we construct a $3\times 3$ commutative diagram whose rows and columns are exact sequences of Hopf algebras. In particular, our approach provides a unified treatment of the Drinfeld doubles associated with restricted Laistrygonian algebras in arbitrary positive characteristic.
\end{abstract}
	
	\maketitle

	\setcounter{tocdepth}{1}
\tableofcontents


\section*{Introduction}
The Drinfeld double is a fundamental construction in the theory of Hopf algebras. Introduced by Drinfeld in the 1980s, it provides a systematic method for constructing quasitriangular Hopf algebras and, consequently, solutions of the quantum Yang--Baxter equation. Drinfeld doubles arise naturally in the study of quantum groups, braided tensor categories and the representation theory of Hopf algebras. In particular, the Drinfeld doubles of bosonizations of Nichols algebras provide a natural framework for constructing and studying Hopf algebras that share many structural features with quantum groups.

Nichols algebras play a central role in the classification of pointed Hopf algebras. Over fields of characteristic zero, several families of Nichols algebras of finite Gelfand--Kirillov dimension over abelian groups were obtained in the classification program developed in \cite{AAH-triang}. Among the families appearing in this classification are the Nichols algebras associated with braided vector spaces called blocks. Given $\epsilon \in \Bbbk^\times$ and $\ell \geq 2$, let $\mathcal{V}(\epsilon,\ell)$ denote the braided vector space with basis $\{x_i:1\leq i\leq \ell\}$ and braiding
\[
\mathfrak{c}_{\epsilon}(x_i \otimes x_j)=
\begin{cases}
	\epsilon x_1\otimes x_i, & j=1,\\
	(\epsilon x_j+x_{j-1})\otimes x_i, & 2\leq j\leq \ell,
\end{cases}
\qquad 1\leq i\leq \ell.
\]
Over a field of characteristic zero, the Nichols algebra of a block
$\mathcal V(\epsilon,\ell)$ has finite Gelfand--Kirillov dimension if and only if
$\ell=2$ and $\epsilon=\pm1$; see \cite{AAH-triang}. Precisely, the Nichols algebra $\toba(\mathcal V(1,2))$ is the Jordan plane and $\toba(\mathcal V(-1,2))$ is the super Jordan plane.

Over fields of odd positive characteristic $p$, the Nichols algebras of the blocks $\mathcal V(1,2)$ and $\mathcal V(-1,2)$ were determined in \cite{clw} and \cite{AAH-findim}, respectively. These algebras are known as the \emph{restricted Jordan plane} and the \emph{restricted super Jordan plane}, respectively. The restricted Jordan plane (resp. super Jordan plane) is a Hopf algebra in the category of Yetter--Drinfeld modules over the cyclic group $C_p$ (resp. $C_{2p}$). The Drinfeld double $D(\toba(\mathcal V(1,2))\#\Bbbk C_p)$ was determined in \cite{APenha}, whereas the Drinfeld double $D(\toba(\mathcal V(-1,2))\#\Bbbk C_{2p})$ was calculated in \cite{APenha2}. In both papers, it is shown that the corresponding Drinfeld double fits into an abelian exact sequence of Hopf algebras. This exact sequence is used to compute the finite-dimensional simple modules over the corresponding Drinfeld double.

Over a field of characteristic $2$, the Nichols algebra of the block $\mathcal V(1,2)$ was determined in \cite{clw}. This algebra is a braided Hopf algebra in the category of Yetter--Drinfeld modules over the cyclic group $C_2$. The corresponding Drinfeld double $D(\toba(\mathcal V(1,2))\#\Bbbk C_2)$ was determined and studied in \cite{A-Bagio-FF3}. As in the odd characteristic case, the Drinfeld double fits into an abelian exact sequence of Hopf algebras, and this exact sequence is used to determine its finite-dimensional simple modules.

The preceding examples correspond to Nichols algebras associated with two-dimensional blocks. A natural next step is to consider braided vector spaces obtained by adjoining a point to a Jordan block. This construction gives rise to the family of braided vector spaces introduced in \cite{AAH-triang}. Let $a,q\in\Bbbk$, $q\neq 0$. The braided vector space $\Lstr_p(q,a)$ is the vector space with basis $\{x_i:i\in \I_3\}$ and braiding $c$ given by
\begin{align*}
	(c(x_i \otimes x_j))_{i,j\in \I_3} &=
	\begin{pmatrix}
		x_{1} \otimes x_{1}& (x_2 + x_{1}) \otimes x_{1}& q x_3 \otimes x_{1}
		\\
		x_{1} \otimes x_2 & (x_2 + x_{1}) \otimes x_2& q x_3 \otimes x_2
		\\
		q^{-1} x_{1} \otimes x_3 & q^{-1}(x_2 + a x_{1}) \otimes x_3& x_3 \otimes x_3
	\end{pmatrix}.
\end{align*}
It was proved in \cite{AAH-triang} that the Nichols algebra $\mathfrak{B}(\Lstr_p(q,a))$, called the \emph{Laistrygonian algebra}, has finite Gelfand--Kirillov dimension over fields of characteristic zero. The Laistrygonian algebras were further studied in \cite{ABFF}, where, among other results, their finite-dimensional simple modules and point modules were classified. More recently, T. Lu \cite{Tao} continued the study of their ring-theoretic properties, investigating their centers, skewfields of fractions, algebra isomorphisms and prime spectra.

Over fields of positive characteristic, finite-dimensional restricted analogues of the Laistrygonian algebra were constructed in \cite{AAH-findim} for $p>2$ and in \cite{A-Bagio-FF2} for $p=2$. In the present work, we determine the Drinfeld doubles of the bosonizations of these restricted Laistrygonian algebras and study their structure and finite-dimensional representations. In contrast to previous studies of the Drinfeld doubles of other restricted Nichols algebras, where the cases $p=2$ and $p>2$ were treated separately, as in \cite{A-Bagio-FF3} and \cite{APenha}, respectively, our approach allows us to treat all positive characteristics simultaneously.

One remarkable feature of the algebra obtained as the Drinfeld double of the bosonization of a restricted Laistrygonian algebra, in contrast with previous work on the Jordan and super Jordan planes, is that the classical Lie algebra arising as part of the abelian extension is not a contragredient Lie algebra. These examples thus provide a possible bridge towards the construction of quantum groups over certain non-semisimple Lie algebras.

The paper is organized as follows. In Section~\ref{sec:preliminaries}, we recall the necessary background on Drinfeld doubles, Nichols algebras, and bosonizations. In Section~\ref{sec:laistrygonian}, we review the Laistrygonian algebra, its restricted versions, and the corresponding Hopf algebras. In Section~\ref{sec:double-restricted}, we determine the Drinfeld doubles of the restricted Laistrygonian algebras, provide presentations by generators and relations, and prove that they are pointed Hopf algebras. In Section~\ref{sec:double-infinite}, we introduce an infinite-dimensional Hopf algebra covering the restricted Drinfeld double. In Section~\ref{sec:extensions}, we describe exact sequences of Hopf algebras associated with $\mathtt{D_p(1,a)}$, identify a restricted enveloping algebra quotient, and construct a commutative $3\times3$ diagram whose rows and columns are exact. In Section~\ref{section:simplemodules}, we classify the finite-dimensional simple $\mathtt{D_p(1,a)}$-modules. Finally, Appendix~\ref{sec:formulas} collects the commutation and comultiplication formulas used throughout the paper.

\section{Preliminaries}\label{sec:preliminaries}
\subsection*{Notations and Conventions}
Throughout, $\ku$ is an algebraically closed field of characteristic $p$, where either $p=0$ or $p$ is a prime number. All
vector spaces, algebras and tensor products are over $\ku$.
Given $\ell < m \in\N_0$,  we set $\I_{\ell, m}=\{\ell, \ell +1,\dots,m\}$, $\I_m = \I_{1, m}$.
Let $\cV$ be a finite-dimensional vector space. We identify $\cV^*\ot \cV^*\simeq (\cV\ot \cV)^*$ via
\begin{align}\label{eq:dual-tensor-product}
	\langle f\ot f', x\ot y \rangle &= \langle f,x\rangle \langle f',y \rangle, & f,f'\in \cV^*, x,y\in \cV.
\end{align}
Given a subset $X$ of a vector space $V$, we denote its linear span by $\Bbbk\{X\}$. We will denote by  $A^\circ=\{f\in A^* \,:\, \ker f \text{ contains an ideal of finite codimension}\}$ the Sweedler dual of an infinite-dimensional algebra $A$.

The cyclic group of order $N$  is denoted by $C_N$ and written multiplicatively. 
The group of $N$--th roots of unity is denoted by  $\G_N$; $\G_N'\subseteq \G_N$ is the subset of primitive \(N\)-th roots of unity.
The algebra of regular functions on an (affine) algebraic group $G$ is denoted $\cO(G)$.
As usual, $\Gb_a$ is the additive algebraic group $(\ku,+)$ and $\Gb_m$ is the multiplicative algebraic group $(\kut,\cdot)$.

Let $A$ be an algebra. If $A$ is generated by $(a_i)_{i \in \I_n}$ subject to the relations $(r_i)_{i \in \I_m}$, then we write  $A=\ku\langle a_1, \dots,a_n\,  \vert\, r_1, \dots,r_m \rangle$.
We denote by $\GK A$ the Gelfand-Kirillov dimension of $A$.
The set of isomorphism classes of finite-di\-men\-sional simple modules over $A$
is denoted $\irr A$. We use the same notation for a class in $\irr A$ and any of its representatives. 
An element $x\in A$ is normal if $Ax = xA$. 
If $B$ is a subalgebra of $A$,
then $\Ind_{B}^{A}$ denotes the induction functor $M \mapsto A \otimes_B M$.

Let $K$ be a Hopf algebra with comultiplication $\Delta$ (we use the Heyneman-Sweedler notation), 
counit $\varepsilon$ and  antipode
$\Ss$. The group of group-like elements is denoted $G(K)$; if $g,h\in G(K)$, then the
space of $(g,h)$ skew-primitive elements of $K$ is denoted by $\Pc_{g,h}(K)$, that is, $\Pc_{g,h}(K):=\{x\in K\,:\,\Delta(x)=  x\ot h + g\ot x\}$.
In particular, $\Pc(K) = \Pc_{1,1}(K)$ is the space of primitive elements. The category of Yetter-Drinfeld modules over $K$ is denoted by $\yd{K} $, see e.~g. \cite[\S 11.6]{Radford-libro}. We refer the reader to \cite{Montgomery-book} and \cite{Radford-libro} for the basic theory  of Hopf algebras needed in this paper.  As usual, \(\delta_{i,j}\) denotes the Kronecker delta. We denote by $\mathbb{G}_N$ the group of roots of unity of order $N$ and $\mathbb{G}_{\infty}=\bigcup_{N\geq 1}\mathbb{G}_N$.

\subsection*{Index of symbols}

For the reader's convenience, we collect below the notation for the main
objects considered throughout the paper. Symbols used only locally are
introduced in the corresponding sections. \medbreak

\begingroup
\renewcommand*{\descriptionlabel}[1]{%
	\hspace{\labelsep}\normalfont #1\hfil}
\begin{description}[leftmargin=!,labelwidth=3.5cm,font=\normalfont]

	\item[\(\Lstr_p(q,a)\)]
	The braided vector space with basis \(\{x_1,x_2,x_3\}\) and braiding
	given in~\eqref{braid-las}.
	
	\item[\(\cL_p(q,a)\)]
	The Laistrygonian algebra associated with \(\Lstr_p(q,a)\).
	
	\item[\(\lstr_p(q,a)\)]
	The restricted Laistrygonian algebra, namely the quotient of
	\(\cL_p(q,a)\) by the truncation ideal \(J_p\).

	\item[\(H_p(q,a)\)]
	The infinite-dimensional bosonization
	\(\cL_p(q,a)\#\ku\Z^2\).
	
	\item[\(\mathtt{H_p(q,a)}\)]
	The finite-dimensional bosonization
	\(\lstr_p(q,a)\#\ku\Gamma\).
	
	\item[\(\Lstr_p^*(q,a)\)]
	The braided dual of \(\Lstr_p(q,a)\).
	
	\item[\(\cL_p^*(q,a)\)]
	The dual Laistrygonian algebra.
	
	\item[\(\lstr_p^*(q,a)\)]
	The restricted dual Laistrygonian algebra.

	\item[\(\mathtt{D_p(q,a)}\)]
	The finite-dimensional Drinfeld double
	\(D\bigl(\mathtt{H_p(q,a)}\bigr)\).
	
	\item[\(K_p(q,a)\)]
	The bosonization \(\cL_p^*(q,a)\#\cA\), where
	\(\cA=U(\mathfrak t)\ot\ku\Z^2\).
	
	\item[\(\DLstr_p(q,a)\)]
	The infinite-dimensional Hopf algebra
	\(H_p(q,a)\bowtie_\tau K_p(q,a)^{\operatorname{op}}\), which covers
	\(\mathtt{D_p(q,a)}\).

	\item[\(\mathfrak{fsl}_2\), \(\mathfrak m\)]
	The fake \(\mathfrak{sl}_2\) and its minimal \(p\)-envelope,
	respectively.
	
	\item[\(\mathfrak a_\ell\)]
	The abelian Lie algebra with basis
	\(\{s_0,\ldots,s_\ell,r_0,\ldots,r_\ell\}\).
	
	\item[\(\mathfrak h_0\), \(\mathfrak h\)]
	The Lie algebras
	\(\mathfrak h_0=\mathfrak{fsl}_2\oplus\ku\{c\}\) and
	\(\mathfrak h=\mathfrak m\oplus\ku\{c\}\).
	
	\item[\(\mathfrak g_0\), \(\mathfrak g\)]
	The semidirect products
	\(\mathfrak g_0=\mathfrak a_\ell\rtimes\mathfrak h_0\) and
	\(\mathfrak g=\mathfrak a_\ell\rtimes\mathfrak h\).

\end{description}
\endgroup

\subsection{The Drinfeld double}\label{subsec:double}
Let $A$ be a finite-dimensional Hopf algebra. The Drinfeld double of $A$, denoted by $D(A)$, was introduced by Drinfeld as a systematic way to construct solutions of the quantum Yang-Baxter equation.
We recall the definition. As a coalgebra, \(D(A)=A\ot A^{*\operatorname{op}}\), with multiplication defined as follows. We write $a\bowtie r:= a \otimes r$ in $D(A)$. Then, 
\begin{align*}
	(a\bowtie r)(a'\bowtie r') &= \big\langle r_{(1)},a'_{(1)}\big\rangle \big\langle r_{(3)}, \Ss(a'_{(3)}) \big\rangle (aa'_{(2)}\bowtie r'r_{(2)}), & 
\end{align*}
$a, a'\in A$, $r, r'\in A^{* \operatorname{op}},$
where $r r' = m(r\ot r')$ is  in $A^*$, not in $A^{* \operatorname{op}}$.

When $A$ is not finite-dimensional, the definition of the double may be extended. 
Following  \cite{Doi-Takeuchi}, the double is defined
with respect to a Hopf algebra $B$ equipped with
a skew-pairing,  i.e., a linear map 
$\tau\colon A\otimes B\to \ku$ satisfying
\begin{align}\label{eq:skew-pairing}
	\begin{aligned}
		\tau(a\widetilde a\ot b) &= \tau(a \ot b_{(1)}) \tau(\widetilde a \ot b_{(2)}),&
		\tau(1 \ot b) &= \varepsilon(b), & a,\widetilde a &\in A,
		\\
		\tau(a \ot \widetilde bb) &= \tau(a_{(1)} \ot b) \tau(a_{(2)} \ot \widetilde b),&
		\tau(a \ot 1) &= \varepsilon(a), &b, \widetilde b &\in B.
	\end{aligned}
\end{align}
Let $\sigma\colon (A \ot B)\ot (A \ot B)\to  \ku$ be the
2-cocycle  associated to $\tau$, where $A \ot B$ has the structure of tensor product Hopf algebra, that is,
\begin{align*}
	\sigma(a\ot b, \widetilde{a}\ot \widetilde{b}) &= \varepsilon(a)\varepsilon(\widetilde{b}) \tau(\widetilde{a}\ot b), & a,\widetilde{a} &\in A, b,\widetilde{b}\in B.
\end{align*}
Then the double of $A$ (with respect to $B$ and $\tau$) is the
Hopf algebra $A\ot B$ twisted by $\sigma$, i.e., $A \bowtie_{\tau} B:= (A\ot B)_{\sigma}$.

Finally, we recall that when \(A\) is finite-dimensional, the category $\yd{A}$ is equi\-valent to the category of left $D(A)$-modules.

\subsection{Nichols algebras} A braided vector space is a pair $(\cV, c)$, where $\cV$ is a vector space and
$c \in GL(\cV^{\otimes 2})$ satisfies the braid equation 
\begin{align*}
	(c\ot \id)(\id\ot c)(c\ot \id) = (\id\ot c)(c\ot \id)(\id\ot c).
\end{align*}  
The braid group $\mathbb B_n$ acts on $\cV^{\otimes n}$, $n \geq 2$. Let $\Omega_n \in \ku \mathbb B_n$ be  the quantum symmetrizer. The Nichols algebra $\toba(\cV)$ is defined as 
$T(\cV)/\left(\oplus_{n \geq 2} \ker \Omega_n\right)$, see \cite{A-leyva} for details.
If $\cV$ is finite-dimensional, then the dual vector space $\cV^*$ is braided with the transpose braiding $c^*$ defined by
\begin{align}\label{eq:braiding-transpose}
	\langle c^*(f\ot f'), x\ot y \rangle &= \langle f\ot f', c(x\ot y)\rangle, & f,f'\in \cV^*, x,y\in \cV.
\end{align}
A realization of $(\cV, c)$ over a Hopf algebra $L$ is a Yetter-Drinfeld $L$-module  structure on $\cV$
such that $c$ coincides with the categorical braiding. Then, $\toba(\cV)$ is a Hopf algebra in $\yd{L}$ and the bosonization $\toba(\cV) \# L$ is a Hopf algebra.

\subsection{The dual and the double of a bosonization}\label{subsec:dual-doble-bosonization}
Let $L$ be a finite-dimensional Hopf algebra and $\cV\in\yd{L}$ such that $\dim\toba(\cV)<\infty$.
In this subsection, we explicitly describe the double of the bosonization
$A \coloneqq \toba(\cV)\# L$.
The braided tensor categories $\yd{L}$ and $\yd{L^*}$ are equivalent via the functor 
$F\colon \yd{L}\longrightarrow\yd{L^*}$
defined as follows. If $\cX\in \yd{L}$, then $F(\cX) = \cX$ with the Yetter-Drinfeld structure
\begin{align*}
	f\rightharpoonup v &= \langle f, \Ss(v_{(-1)}) \rangle v_{(0)}, & \rho(v) &= \sum_{i=1}^m \Ss^{-1}(f_i)\ot h_i\rightharpoonup v, & v\in \cX, f\in L^*,
\end{align*}
see \cite[Proposition 2.2.1]{A-Granha}. Here $(h_i)_{i \in \I_m}$ is a basis of $L$ 
and $(f_i)_{i \in \I_m}$ is its dual basis.

\begin{lemma}\label{lemma:functor-dual-hopf} \cite[Lemma 2.6]{Granha}
	Let $\cW\coloneqq F(\cV^*) \in\yd{L^*}$. The braided vector spaces  $\cW$  and $(\cV^*,c^*)$ are isomorphic and
	$\dim\toba(\cW) = \dim\toba(\cV) <\infty$. \qed
\end{lemma}

\begin{lemma}\label{lemma:dual-of-a-bosonization} \cite[Proposition 1.5]{APenha2}
	$A^* \simeq \toba(\cW)\# L^*$. \qed
\end{lemma}

\medbreak
Given bases $(v_i)_{i \in \I_{n}}$ of $\cV$ and $(w_i)_{i \in \I_{n}}$  of $\cW$,
and generators $(s_j)_{j \in \I_{m_0}}$ of $L$ and $(s'_k)_{k\in \I_{m_1}}$  of $(L^*)^{\operatorname{op}}$,
we fix  presentations for $\toba(\cV)$, $\toba(\cW)^{\operatorname{op}}$ and $D(L)$:
\begin{align}\label{eq:general-presentation-by-generators-and-relations-for-Nichols}
	\begin{aligned}
		&	\toba(\cV) = \ku \langle v_1,\dots,v_n \,\vert\, r_1,\dots,r_{n_1}\rangle, \\
		&	\toba(\cW)^{\operatorname{op}} = \ku \langle w_1,\dots,w_n\, \vert\, r'_1,\dots,r'_{n_2}\rangle, 
		\\
		&	D(L) =\ku \langle s_1,\dots,s_{m_0}, s'_1,\dots, s'_{m_1}\, \vert\, \widetilde{r}_1,\dots,\widetilde{r}_{n_3}\rangle.
	\end{aligned}
\end{align}  
Consequently, we obtain the following presentation for $D(A)$.

\begin{prop}\label{prop:the-double-in-general}  \cite[Proposition 1.7]{APenha2}
	The algebra $D(A)$ is presented by generators $(v_i)_{i \in \I_{n}}$, $(w_i)_{i \in \I_{n}}$, 
	$(s_j)_{j \in \I_{m_0}}$  and $(s'_k)_{k\in \I_{m_1}}$ with relations \eqref{eq:general-presentation-by-generators-and-relations-for-Nichols} and
	\begin{align}\label{eq:relations-doble-general}
		\begin{aligned}
			&	s'_i v_j = \langle (s'_i)_{(1)}, (v_j)_{(-1)} \rangle (v_j)_{(0)} (s'_i)_{(2)}, \quad (i,j)\in\I_{m_1}\times \I_n, \\
			&	w_i s_j = \sum_{\ell=1}^n \langle w_i, (s_j)_{(1)}\rightharpoonup v_\ell\rangle (s_j)_{(2)} w_\ell, \quad (i,j)\in\I_{n}\times \I_{m_0}, \\
			&	w_i v_j = \langle w_i, v_j\rangle 1 + \sum_{\ell=1}^n \langle w_i, (v_j)_{(-1)}\rightharpoonup v_\ell\rangle (v_j)_{(0)} w_\ell
		\end{aligned}\\ \notag
		\begin{aligned}
			+\sum_{k=1}^m \sum_{\ell,t=1}^n &\langle w_i, (v_j)_{(-2)}\rightharpoonup v_\ell\rangle \langle w_\ell, h_k\rightharpoonup v_t \rangle
			\langle w_t\#\varepsilon, \Ss((v_j)_{(0)}) \rangle (v_j)_{(-1)} f_k,
		\end{aligned}
	\end{align}
	$(i,j)\in\I_n\times \I_n$. 
	Hence $D(A)$ admits a triangular decomposition 
	\begin{align*}
		\hspace{88pt} D(A) \simeq \toba(\cV)\ot D(L) \ot \toba(\cW)^{\operatorname{op}}.  \hspace{88pt}\qed
	\end{align*}
\end{prop}


\section{The Laistrygonian algebra and its restricted versions}\label{sec:laistrygonian}

Let $a, q\in\ku^\times$ and $\Lstr_p(q,a)$ the braided vector space with basis $\{x_i: i\in \I_3\}$ and braiding $c$ given by
\begin{align}\label{braid-las}
	(c(x_i \otimes x_j))_{i,j\in \I_3} &=
	\begin{pmatrix}
		x_{1} \otimes x_{1}&  ( x_2 + x_{1}) \otimes x_{1}& q x_3  \otimes x_{1}
		\\
		x_{1} \otimes x_2 & ( x_2 + x_{1}) \otimes x_2& q x_3  \otimes x_2
		\\
		q^{-1} x_{1} \otimes x_3 &  q^{-1}(x_2 + a x_{1}) \otimes x_3&  x_3  \otimes x_3
	\end{pmatrix}.
\end{align}

In this paper we assume that:
\begin{itemize}[label=$\circ$]
	\item If  $p>0$, then $q\in\G_N'$ with $\gcd(p,N)=1$.  \vspace{.1cm}
	\item If $p  \neq 2$, then we take a normalization of $a$ in the following sense: given $\ghost\in \N$, we choose the scalar $a\in\ku^{\times}$ such that $-2a = \ghost$. If $p>2$ we also assume that $\ghost\in \I_{p-1}$.
\end{itemize}

\subsection{The Laistrygonian algebra}

The Laistrygonian algebra was originally defined  in \cite{AAH-triang} and \cite{AAH-findim} over fields of characteristic different from $2$. In order to provide a unified treatment, we extend the definition to the case $p=2$. Set
\begin{align}
	\ell =
	\begin{cases}
		\ghost, & \text{if } p \neq 2,\\
		3-\delta_{a,1}, & \text{if } p = 2.
	\end{cases}
\end{align}
Let  $T(\Lstr_p(q,a))$ denote the tensor algebra. Define
\begin{align*}
	\xb_{21}=\xb_2\xb_1-\xb_1\xb_2, &&
	z_0= x_3, &&
	z_{n+1} = \xb_2 z_n - q z_n \xb_2,
	& & n\in \mathbb{N}_0.
\end{align*}

\begin{definition}
	The \emph{Laistrygonian algebra} $\cL_p(q,a)$ is the algebra generated by $\xb_i, z_0$,  $i\in \I_2$, subject to the relations
	\begin{equation}\label{eq: relation-lestrygon-geral-1}
		\begin{aligned}
			\xb_1^2 =-2\xb_{21},&\quad &
			\xb_1 z_n = q z_n \xb_1, &\quad &
			z_{\ell+1}=0,&\quad &
			\xb_{21}\xb_2 = (\xb_2+\xb_1 )\xb_{21}. 
		\end{aligned}
	\end{equation}
\end{definition}

\begin{remark} If $p \neq 2$, then the presentation given in \eqref{eq: relation-lestrygon-geral-1} is equivalent to the original one in \cite{AAH-triang} and \cite{AAH-findim}, namely
	\begin{align}\label{eq: relation-lestrygon-geral-1Original}
		\xb_1 \xb_2&= \xb_2 \xb_1 + \tfrac{1}{2}\xb_1^2,&
		\xb_1 z_0&= q z_0 \xb_1,&
		z_{\ell+1}&=0.
	\end{align}
	Indeed, the  relation $\xb_{21}\xb_2 = (\xb_2+\xb_1 )\xb_{21}$ becomes redundant in this case, and $\xb_1 z_n= q z_n \xb_1$ follows from the case $n=0$ and the first relation of \eqref{eq: relation-lestrygon-geral-1Original}. On the other hand, if $p=2$, the relations in \eqref{eq: relation-lestrygon-geral-1} are equivalent to
	\begin{equation}\label{lestrygon-char2}
		\begin{aligned}
			&\xb_1^2 =0, &\quad 
			&\xb_1\xb_2^2 =\xb_2^2\xb_1 + \xb_1\xb_2\xb_1, & \quad 
			&\xb_1 z_n = q z_n \xb_1,& \quad 
			&z_{4-\delta_{a,1}}=0.&
		\end{aligned}
	\end{equation}
In fact, the second relation in \eqref{lestrygon-char2} follows from $\xb_{21}\xb_2 = (\xb_2+\xb_1 )\xb_{21}$. The others are immediate. 
\end{remark}

When \(p=0\), we recall the following result from \cite[Proposition 4.3.5]{AAH-triang}.

\begin{lemma} \label{lem-nichol-lest-ch0}  $\cL_{0}(q,a) \simeq \toba(\Lstr_0(q,a))$ and $\GK\cL_{0}(q,a) = \ghost +3$. \qed
\end{lemma}


\subsection{The restricted Laistrygonian algebra}
Over fields of positive charac\-te\-ris\-tic, the Laistrygonian algebra admits restricted versions defined in  \cite{AAH-findim} and  \cite{A-Bagio-FF2} for $p\neq 2$ and  $p=2$, respectively. As in the case of characteristic zero, this algebra is related to the braided vector space introduced at the beginning of this section. 
\begin{definition} Let $p > 0$, and	let $J_p$ the ideal of $\cL_p(q,a)$ generated by
	\begin{align}\label{eq:trunc-Lestrygon-char-p-arbitraria}
		\xb_1^p,\qquad	 \xb_{21}^p,\qquad   \xb_2^{p(1+\delta_{2,p})},\qquad  z_n^p,\qquad  n\in \I_{0, \ell}.  
	\end{align} The \emph{restricted Laistrygonian algebra} $\lstr_p(q,a)$ is the quotient of  the algebra $\cL_p(q,a)$ by the ideal $J_p$ 
\end{definition}

\begin{remark}
Since $x_1^p=0$, it follows that $x_{21}^p=0$ for $p\neq 2$. On the other hand, if $p=2$, then the presentation given by \eqref{lestrygon-char2} and \eqref{eq:trunc-Lestrygon-char-p-arbitraria} coincides with the original presentation in \cite{A-Bagio-FF2}, since the relation $x_{21}^2=0$ is equivalent to $x_1x_2x_1x_2=x_2x_1x_2x_1$.
\end{remark}

\begin{lemma} \label{lem:lest-rest-nichols}
	If $p>0$, then  $\lstr_p(q,a) \simeq \toba(\Lstr_p(q,a))$ is a quotient of  $\cL_p(q,a)$ and has dimension $p^{\ell + 3 + 2\delta_{2,p}}$.
\end{lemma}
\begin{proof}
	The case $p > 2$ was proven in \cite[Proposition 4.10]{AAH-findim}, while the case for $p=2$ was established in \cite[Propositions 3.6 and 3.7]{A-Bagio-FF2}.
\end{proof}

\subsection{The corresponding Hopf algebras}
Let  $\{\alpha_1,\alpha_2\}$ be the canonical basis of $\Z^2$. The braided vector space
$\Lstr_p(q,a)$ has a realization in $\yd{\ku\Z^2}$ with action  and coaction  given by
\begin{align}\label{eq:YDlestrygon-action-coaction}
	\begin{aligned}
		&	 \alpha_1\rightharpoonup \xb_i = (\xb_i + \delta_{i,2} \xb_1), & &	\rho(\xb_1) = \alpha_1 \ot \xb_1,\\
		& \alpha_2\rightharpoonup \xb_i = q^{-1}(\xb_i + a\,  \delta_{i,2}  \xb_1), & &\rho(\xb_2) = \alpha_1 \ot \xb_2,\\
		& \alpha_j\rightharpoonup z_0 = q^{\delta_{1,j}} z_0, & &\rho(z_0) = \alpha_2 \ot z_0.
	\end{aligned}
\end{align} 
Thus $\cL_p(q,a)$ is a Hopf algebra in $\yd{\ku \Z^2}$. For $p>0$, consider the group $\varGamma = C_{pN}\times C_{pN}= \langle \alpha_1 \rangle \times \langle \alpha_2 \rangle$.
The formulas 
\eqref{eq:YDlestrygon-action-coaction} also provide  a realization of $\Lstr_p(q,a)$   
in $\yd{\ku \varGamma}$ and consequently, $\lstr_p(q,a)$ is a  Hopf algebra in $\yd{\ku \varGamma}$. 
Let
\begin{align}
	\label{eq:def-H}	H_{p}(q,a)&\coloneqq \cL_{p}(q,a)\# \ku\Z^2, &   \mathtt{H_p(q,a)}&\coloneqq \lstr_p(q,a)\# \ku \varGamma.
\end{align}

Observe that  $\mathtt{H_p(q,a)}$ is finite-dimensional, while $H_{p}(q,a)$ is infinite-dimensional but with finite $\GK$ension  (see Proposition \ref{prop:Dtilde} below).

\begin{lemma}\label{lemma:presentacion-H}
	\begin{enumerate}[leftmargin=*,label=\rm{(\roman*)}]
		\item	The algebra $H_p(q,a)$ is generated by $\xb_i, z_0,  \alpha_i$, $i\in \I_2$, with relations \eqref{eq: relation-lestrygon-geral-1} and
		\begin{align}\label{eq: relations H}
			\begin{aligned}
				&\alpha_1 \xb_i = (\xb_i + \delta_{2,i}\, \xb_1) \alpha_1, & &\qquad 	 \alpha_2 \xb_i = q^{-1} (\xb_i + a\, \delta_{2,i}\,  \xb_1)\alpha_2,\\
				&\alpha_i z_0 = q^{\delta_{1,i}} z_0 \alpha_i& &\qquad \alpha_1 \alpha_2 = \alpha_2 \alpha_1.                       
			\end{aligned}
		\end{align}
		\item  The algebra $\mathtt{H_p(q,a)}$ is generated by $\xb_i, z_0,  \alpha_i$, $i \in \I_2$ with relations, \eqref{eq: relation-lestrygon-geral-1}, \eqref{eq:trunc-Lestrygon-char-p-arbitraria}, \eqref{eq: relations H} and 
		\begin{align}\label{eq: alphaNp} \alpha_i^{Np}= 1.\end{align}
	\end{enumerate}
	\noindent The comultiplication is determined by
	\begin{equation}
		\begin{aligned}\label{eq: Delta H}
			&x_i \in \Pc_{\alpha_1,1}(H_p(q,a)),&
			&z_0 \in \Pc_{\alpha_2,1}(H_p(q,a)), &\alpha_i \in G(H_p(q,a)),&
		\end{aligned}
	\end{equation}
	and analogously for $\mathtt{H_p(q,a)}$.\qed
\end{lemma}

\subsection{Finite generation of cohomology}
\label{subsec:fgc}

Recall that an augmented algebra \(A\) is said to have finitely
generated cohomology, or simply \emph{fgc}, if the graded algebra
\(\operatorname{Ext}_{A}^{\bullet}(\ku,\ku)\) is finitely generated
and, for every finitely generated \(A\)-module \(M\), the graded
\(\operatorname{Ext}_{A}^{\bullet}(\ku,\ku)\)-module
\(\operatorname{Ext}_{A}^{\bullet}(\ku,M)\) is finitely generated.

Assume that \(p>2\) and let \(q\in\mathbb{G}_{\infty}\). Set
\(N=\ord(q)\). The notation \(\Lstr_q(1,\ghost)\) used in
\cite[Section~5]{AN-fgc} corresponds to our braided vector space
\(\Lstr_p(q,a)\) under the identification
\[
x_1\longleftrightarrow x_1,\qquad
y_1\longleftrightarrow x_2,\qquad
x_2\longleftrightarrow x_3,\qquad
\ghost=-2a.
\]
Moreover, the realization considered therein can be taken over
\(\Gamma=C_{pN}\times C_{pN}\). Hence
\cite[Theorem~5.4]{AN-fgc} shows that the bosonization
\[
\mathtt{H_p(q,a)}
=\lstr_p(q,a)\#\ku\Gamma
\]
has fgc, whereas \cite[Corollary~5.5]{AN-fgc} shows that the
restricted Laistrygonian algebra \(\lstr_p(q,a)\) itself has fgc. In the particular case \(q=1\), one has \(N=1\) and
\(\Gamma=C_p\times C_p\). In this case,
\cite[Proposition~5.3]{AN-fgc} also shows that the Drinfeld double (see Proposition \ref{prop:double-both}) $\mathtt{D_p(1,a)}=D\bigl(\mathtt{H}_p(1,a)\bigr)$ has fgc.


\section{Drinfeld doubles of restricted Laistrygonian algebras} \label{sec:double-restricted}
Throughout this section, we consider $p>0$. Our goal is to compute the Drinfeld doubles of the Hopf algebras $\mathtt{H_p(q,a)}$.

\subsection{The double of the group algebra} As in the previous section, consider $\varGamma = C_{pN}\times C_{pN}$. Then, $\ku^{\varGamma} \simeq \ku^{\varGamma_1} \otimes \ku^{\varGamma_2}$,  where
$\varGamma_1 \simeq C_{p}\times C_{p}$ and $\varGamma_2 \simeq C_{N}\times C_{N}$. Let $\zeta_i \in \ku^\varGamma$, $i \in \I_2$,  the derivations defined by
\begin{align}\label{eq:pairing-G-HGdual-ders}
& \langle \zeta_i, \alpha_j\rangle = \delta_{ij}, \quad  i,j \in \I_2.
\end{align}
Note that $\langle \zeta_1, \alpha_1^r\alpha_2^s\rangle=r1_{\ku}$, $\langle \zeta_2, \alpha_1^r\alpha_2^s\rangle=s1_{\ku}$ and consequently $\zeta_i^p=\zeta_i$, $i\in \I_{2}$.
Thus, $ \ku^{\varGamma_1}\simeq \ku[\zeta_1,\zeta_2]/\langle \zeta_i^p-\zeta_i: i \in \I_2 \rangle $ is the  restricted enveloping algebra of the $2$-dimensional abelian
Lie algebra. Let $T_i \in \ku^\varGamma$, $i \in \I_2$, the characters defined by
\begin{align}\label{eq:pairing-G-HGdual-chars}
	T_1 (\alpha_1) &= q, & T_1 (\alpha_2) &=1= T_2 (\alpha_1), & T_2 (\alpha_2) &= q^{-1}.
\end{align}
Hence $ \ku^{\varGamma_2}$ is generated by $T_1$ and $T_2$. 
Summarizing we have the next result.
\begin{lemma}\label{lema:double-group}
The double $D(\ku \varGamma) \simeq \ku \varGamma \otimes \ku^\varGamma$ is the  commutative Hopf algebra generated by
	$\alpha_i$,  $\zeta_i$,  $T_i$, $i \in \I_2$, with  relations \eqref{eq: alphaNp} and
	\begin{align}\label{eq:trunc-zetai-Ti}
\zeta_i^p &=\zeta_i,&   T_i^{N} &= 1.
\end{align}
The comultiplication is given by $\alpha_i, T_i \in G(D(\ku \varGamma))$ and $\zeta_i \in \mathcal{P}(D(\ku \varGamma))$.   \qed
\end{lemma}

\subsection{The dual of restricted Laistrygonian algebra} 
We will write $\Lstr^*_{p}(q,a)$ instead of $(\Lstr_p(q,a))^*$. 
Consider the basis $\{\ub_i : i\in \I_3\}$ of $\Lstr^*_{p}(q,a)$ given by: 
	\begin{align}\label{eq:pairing-H-Hdual}
		\langle \ub_1, \xb_2 \rangle &= 1, & \langle \ub_2, \xb_1 \rangle &= 1, & \langle \ub_3, z_0 \rangle &= 1,
	\end{align}
	and the remaining evaluations being $0$. 
	It follows from  \eqref{eq:braiding-transpose}  that the braiding $c^*$  is  given by
		\begin{align*}
		(c^{\ast}(u_i \otimes u_j))_{i,j\in \I_3} \!&=\!
		\begin{pmatrix}
			u_1 \otimes u_1&   u_2  \otimes u_1& q^{-1} u_3 \otimes u_1
			\\
			u_1 \otimes (u_1+u_2) &  u_2  \otimes (u_1+u_2)& q^{-1} u_3  \otimes (u_2+au_1)
			\\
			q u_1 \otimes u_3 &  q u_2 \otimes u_3&  u_3  \otimes u_3
		\end{pmatrix}\!.
	\end{align*}


\begin{lemma}\label{lema: isomorfismo-V-inverse}
	The linear map $\psi\colon(\Lstr_{p}^*(q,a),c^*) \longrightarrow (\Lstr_{p}(q^{-1},a),c^{-1})$ given by
	\[u_1 \mapsto x_1, \quad u_2 \mapsto -x_2,\quad u_3\mapsto z_0,\] is an isomorphism  of braided vector spaces. \qed
\end{lemma}

\begin{definition} \label{def:les-dual}
$\cL^*_{p}(q,a)$  is the algebra generated by $\ub_i, i\in \I_3$, subject to the relations
\begin{align}
	\label{eq: relation-lestrygon-dual1}
&\ub_1^2 =  -2 \ub_{21}, & \quad 
&w_n\ub_1  = q\,  \ub_1 w_n, & \quad 
&w_{\ell+1}=0,  & \quad
&\ub_{2}\ub_{21}  =  \ub_{21}(\ub_2-\ub_1),
\end{align}
where 
\begin{equation*}
\ub_{21}\coloneqq \ub_2\ub_1-\ub_1\ub_2,\qquad  	w_0 \coloneqq \ub_3,\qquad  w_{n+1} \coloneqq \ub_2 w_n - q^{-1} w_n\ub_2,\quad  n\in \N_0.
\end{equation*}
\end{definition}

\begin{lemma}\label{lemma:Nichols dual-lestrigon}
 $\cL^*_{0}(q,a)\simeq \toba(\Lstr^*_{0}(q,a))$ as algebras.
\end{lemma}
\begin{proof}
	It follows from Lemmas \ref{lem-nichol-lest-ch0} and \ref{lema: isomorfismo-V-inverse}.
\end{proof}

As in the case $\lstr_p(q,a)$, we can define the restricted version for $\lstr_p^*(q,a)$.
\begin{definition} Let $J_p^*$ be the ideal of  $\cL_p^*(q,a)$ generated by 
	\begin{align}\label{eq:trunc-Lestrygon-dual-char-p-arbitraria}
	\ub_1^p,\qquad   \ub_{21}^p,\qquad \ub_2^{p(1+\delta_{2,p})},\qquad   w_n^p,\quad   n\in \I_{0, \ell}. 
	\end{align} The \emph{dual restricted Laistrygonian algebra} $\lstr_p^*(q,a)$ is the quotient of  the algebra $\cL_p^*(q,a)$ by the ideal $J_p^*$. 
\end{definition}

\begin{lemma} 
$\lstr_p^*(q,a) \simeq \toba(\Lstr^*_p(q,a))$
is a quotient of  $\cL_p^*(q,a)$ of dimension $p^{\ell + 3 + 4\delta_{2,p}}$.
\end{lemma}
\begin{proof}
It follows from Lemmas \ref{lem:lest-rest-nichols} and \ref{lema: isomorfismo-V-inverse}.
\end{proof}

By  Lemma \ref{lemma:functor-dual-hopf} and \cite[Proposition 2.2.1]{A-Granha}, we realize $\Lstr_{p}^*(q,a)$ in $\yd{\ku^\varGamma}$.
Explicitly, the action is given by:
\begin{align}\label{eq:struct-yd-dual}
\begin{aligned}
&	\zeta_j \rightharpoonup u_i = \delta_{1,j} u_i,  &\qquad &T_j \rightharpoonup u_i = q^{\delta_{1,j}} u_i,,\\
&	\zeta_j \rightharpoonup w_0 = \delta_{2,j} w_0, &\qquad & T_j \rightharpoonup w_0 = q^{-\delta_{2,j}} w_0,
\end{aligned}
\end{align}
$i, j\in \I_2$. Now, denote $\zeta(a):=\zeta_1 + a \zeta_2$. The coaction of $\ku^\varGamma$ on $\Lstr_{p}^*(q,a)$ is given by:
\begin{align}
\rho(u_1) &= T_2 \ot u_1, & \rho(u_2) &= T_2 \ot u_2 + T_2\zeta(a)\ot u_1, & \rho(w_0) &= T_1 \ot w_0.
\end{align}

\subsection{The Drinfeld double of  $\mathtt{H_p(q,a)}$} To present the Drinfeld double of $\mathtt{H_p(q,a)}$, we begin with the following result.

\begin{lemma}\label{lemma:H dual-op-both}
The Hopf algebras $\big(\mathtt{H_p(q,a)}^*\big)^{\rm op}$ and $\big(\lstr_p^*(q,a)\# \ku^\varGamma\big)^{\rm op}$ are isomorphic. Moreover, $\big(\mathtt{H_p(q,a)}^*\big)^{\rm op}$ is presented by generators  $\ub_i$, $w_0$, $\zeta_i$, $T_i$, $i \in \I_2$, satisfying the relations \eqref{eq:trunc-zetai-Ti},   \eqref{eq: relation-lestrygon-dual1}, \eqref{eq:trunc-Lestrygon-dual-char-p-arbitraria} and
\begin{align}\label{eq relations H dual-op1}
	\begin{aligned}
			&	\zeta_1\zeta_2	=\zeta_2\zeta_1,		  & & \zeta_j T_i=T_i\zeta_j,		&  &  T_1T_2 = T_2T_1, \\
 	&\ub_i T_j  = q^{\delta_{1,j}} T_j  \ub_i,& & \ub_i  \zeta_j = \zeta_j   \ub_i  + \delta_{1,j} \ub_i,      & &\\
 	& 	w_0  T_j = q^{-\delta_{2,j}}T_j  w_0, & &	w_0 \zeta_j = \zeta_j w_0+ \delta_{2,j} w_0,   & & i,j \in \I_2.\\
	\end{aligned}
\end{align}	
The comultiplication is determined by
\begin{align}\label{eq: Delta H dual}
	\begin{aligned}
		&T_i\in G((\mathtt{H_p(q,a)}^*)^{\operatorname{op}}),&  &\zeta_i  \in \Pc ((\mathtt{H_p(q,a)}^*)^{\operatorname{op}}),\quad  i \in \I_2,\\ 
		& \ub_1 \in \Pc_{T_2,1}((\mathtt{H_p(q,a)}^*)^{\operatorname{op}}), & 	&w_0 \in \Pc_{T_1,1}((\mathtt{H_p(q,a)}^*)^{\operatorname{op}}),\\
	 & \Delta(\ub_2) = \ub_2\ot 1 + T_2\ot \ub_2 + \zeta(a)T_2 \ot \ub_1.
	\end{aligned}
\end{align} \qed
\end{lemma}

Now we present the Drinfeld double of $\mathtt{H_p(q,a)}$ by generators and relations.

\begin{prop}\label{prop:double-both}
	The Drinfeld double $\mathtt{D_p(q,a)}=D(\mathtt{H_p(q,a)})$ is presented by ge\-ne\-rators   $\xb_i$, $z_0$, $\alpha_i$, $\ub_i$, $w_0$, $T_i$, $\zeta_i$,   $i \in \I_2$, with relations \eqref{eq: relation-lestrygon-geral-1}, \eqref{eq:trunc-Lestrygon-char-p-arbitraria}, \eqref{eq: relations H}, \eqref{eq: alphaNp}, \eqref{eq:trunc-zetai-Ti}, \eqref{eq: relation-lestrygon-dual1}, \eqref{eq:trunc-Lestrygon-dual-char-p-arbitraria},  \eqref{eq relations H dual-op1} and
	
	\begin{equation}\label{eq:double-relations-1} 
	\begin{aligned}
		&T_j \xb_i    = q^{\delta_{1,j}} \xb_i T_j,&  & T_jz_0 = q^{-\delta_{2,j}}z_0 T_j,\\ 
	&\zeta_j \xb_i = \xb_i \zeta_j + \delta_{1,j}\xb_i, & & \zeta_j z_0 = z_0 \zeta_j+\delta_{2,j} z_0,\\ 
	&	\ub_j \alpha_1 = \alpha_1 (\ub_j+\delta_{2,j}\ub_1), & & \ub_j \alpha_2 = q^{-1} \alpha_2 (\ub_j+a\,  \delta_{2,j}\ub_1),\\
		& w_0 \alpha_i = q^{\delta_{i,1}} \alpha_i w_0,& &  \ub_j z_0 = q^{-1} z_0  (\ub_j+a\,  \delta_{2,j}\ub_1),\\
		&\ub_1 \xb_1 = \xb_1 \ub_1, &  &\ub_2 \xb_1 = \xb_1 (\ub_2 + \ub_1) + 1 - \alpha_1 T_2,	         \\
	   &\ub_1 \xb_2 = \xb_2 \ub_1 + 1 - \alpha_1 T_2,       & & \ub_2 \xb_2 = \xb_2 (\ub_2 + \ub_1) -  \alpha_1 T_2\zeta(a),       \\
	&	w_0 \xb_i = q \xb_i w_0, & &	w_0 z_0 = z_0 w_0 + 1 - \alpha_2 T_1 .
	\end{aligned}
\end{equation}	
The comultiplication of $\mathtt{D_p(q,a)}$ is determined by  \eqref{eq: Delta H} and \eqref{eq: Delta H dual}.
\end{prop}
\begin{proof}
		It is a direct application of Proposition \ref{prop:the-double-in-general}.
\end{proof}


\begin{coro}\label{coro:double-pointed2}
	\(\mathtt{D_p(q,a)}\) is a pointed Hopf algebra.\end{coro}
\pf
Consider the following subsets of \(\mathtt{D_p(q,a)}\):
\begin{align*}  &C_{\{0\}}=\ku\{1,\alpha_i,T_i: i \in \I_2\}, & &  C_{\{1\}}=  C_{\{0\}} + \ku\{x_i,z_0, u_1,w_0,\zeta_i,\zeta(a)T_2: i \in \I_2\},\\  
	& C_{\{2\}}= C_{\{1\}}+\{u_2\}. \end{align*}
This is a coalgebra filtration of $C = C_{\{2\}}$.  By \cite[Proposition 4.1.2]{Radford-libro}, the coradical $C_0$ of $C$ is equal to $C_{\{0\}}$.
Since $C$ generates \(\mathtt{D_p(q,a)}\) as algebra, \(\mathtt{D_p(q,a)}_0\subseteq C_{\{0\}}\) by \cite[Corollary 5.1.12]{Radford-libro}. 
Therefore, \(\mathtt{D_p(q,a)}_0= C_{\{0\}}\) and particularly \(\mathtt{D_p(q,a)}\) is pointed. \epf 


\section{Drinfeld doubles of Laistrygonian algebras} \label{sec:double-infinite}
 We introduce an infinite-dimensional Hopf algebra $\mathcal{DL}_p(q, a)$
	covering $\mathtt{D_p(q,a)}$. For this, we define a skew-pairing between two Hopf algebras covering $\mathtt{H_p(q,a)}$ and $(\mathtt{H_p(q,a)}^*)^{\operatorname{op}}$; see \S \ref{subsec:double}. From Lemma \ref{lemma:presentacion-H}, it is clear that  $H_p(q,a)$ described in  \eqref{eq:def-H} covers $\mathtt{H_p(q,a)}$.
	
	Let $\cA \coloneqq U(\mathfrak t)\ot \ku\Z^2$, where $U(\mathfrak t) \simeq S(\mathfrak t)$ 
	is the enveloping algebra of the  $2$-dimensional abelian Lie algebra $\mathfrak t$. 
	The braided  Hopf algebra $\cL_p^*(q,a)$ presented in Definition \ref{def:les-dual} belongs to $\yd{\cA}$ with structure induced by \eqref{eq:struct-yd-dual}. More precisely, assume that $\mathfrak t$ is generated by $\zeta_1, \zeta_2$. The left Yetter-Drinfeld module structure on  $\cL_p^*(q,a)$ is given by \eqref{eq:struct-yd-dual}. Set the bosonization
	\begin{align}\label{alg-aux-cover}
		K_p(q,a) \coloneqq \cL_p^*(q,a)\# \cA.
	\end{align}
 Clearly, $K_p(q,a)^{\operatorname{op}}$ covers $(\mathtt{H_p(q,a)}^*)^{\operatorname{op}}$. In fact, the algebra $K_p(q,a)^{\operatorname{op}}$ has a similar presentation  as  $(\mathtt{H_p(q,a)}^*)^{\operatorname{op}}$ given  in Lemma \ref{lemma:H dual-op-both}, 
	satisfying  the relations \eqref{eq: relation-lestrygon-dual1}, \eqref{eq relations H dual-op1}, \eqref{eq:double-relations-1}  
	 but omitting the set of relations \eqref{eq:trunc-zetai-Ti} and  \eqref{eq:trunc-Lestrygon-dual-char-p-arbitraria}.

	\begin{lemma}\label{lemma:skew-pairing-def-H-K-op}
		The formulas \eqref{eq:pairing-G-HGdual-ders}, \eqref{eq:pairing-G-HGdual-chars} and \eqref{eq:pairing-H-Hdual}
		induce  a skew-pairing between the Hopf algebras $H_p(q,a)$   and $K_p(q,a)^{\operatorname{op}}$.
	\end{lemma} 
\begin{proof}
We sketch a proof. Let $H=H_p(q,a)$ and $K^{\operatorname{op}}=K_p(q,a)^{\operatorname{op}}$. Giving a skew-pairing between $H$ and $K^{\operatorname{op}}$ is equivalent to give a Hopf algebra homomorphism $\tau: H \longrightarrow \left(K^{\operatorname{op}}\right)^\circ$. For each  generator $a$ of $H$ and each generator  $b$ of  $K^{\operatorname{op}}$ let $\tau(a)(b):=\langle a,b\rangle$ as in \eqref{eq:pairing-G-HGdual-ders}, \eqref{eq:pairing-G-HGdual-chars} and \eqref{eq:pairing-H-Hdual}, and extend $\tau(a)$ linearly to a map on $K^{\operatorname{op}}$ by the rule 
\begin{align}\label{eqn: tau multiplicative}
	\tau(a)(b\,  \tilde{b})= \tau(a_{(1)})(\tilde{b})\tau(a_{(2)})(b),\qquad b,\tilde{b}\in K^{\operatorname{op}}.
\end{align}
To prove that $\tau(a)$ is a well-defined  map on $K^{\operatorname{op}}$, we need to check that, for each relation $\mathbf{r}$ of $K^{\operatorname{op}}$, we have  $\tau(a)(\mathbf{r})=0$. By  \eqref{eqn: tau multiplicative}, $m^*(\tau(a))=\tau(a_{(2)})\ot \tau(a_{(1)})$, hence $\tau(a)\in \left( K^{\operatorname{op}}\right)^\circ$ for each generator $a$ of $H$.
Now we have to extend the map $\tau$ to $H^{\operatorname{cop}}$. We do this by the rule		
\begin{align}\label{eqn: tau multiplicative2}
\tau(a \tilde{a})(k)= \tau(a)(k_{(1)})\tau(\tilde{a})(k_{(2)}),\quad a,\tilde{a}\in H,\, \,k\in K^{\operatorname{op}}.
\end{align}
Again, to check that $\tau$ is a well-defined map on $H^{\operatorname{cop}}$, it is enough to check that for every relation $\mathbf{s}$ of $H$, $\tau(\mathbf{s})(b)=0$. By construction, the resulting map \(\tau\) is a skew-pairing.
\end{proof}	

As explained in \S \ref{subsec:double}, we  consider 
\begin{align}\label{alg-cover-positive}
\mathcal{DL}_p(q, a):=H_{p}(q,a)\bowtie_{\tau}K_p(q,a)^{\operatorname{op}},
\end{align}
where $\tau$ is the skew-pairing given in the proof of Lemma \ref{lemma:skew-pairing-def-H-K-op}. As the following proposition shows, $\mathcal{DL}_p(q,a)$ admits a presentation analogous to that in Proposition \ref{prop:double-both} with the truncation relations \eqref{eq:trunc-Lestrygon-char-p-arbitraria}, \eqref{eq: alphaNp}, \eqref{eq:trunc-zetai-Ti} and \eqref{eq:trunc-Lestrygon-dual-char-p-arbitraria} removed.

\begin{prop}\label{prop:Dtilde} The following assertions hold.
\begin{enumerate}[leftmargin=*,label=\rm{(\roman*)}]
		\item\label{prop:Dtilde-uno} The algebra $\mathcal{DL}_p(q, a)$ is generated by $\xb_i, z_0, \alpha_i, \ub_i, w_0,  T_i, 
		\zeta_i$, $i \in \I_2$,
	subject to the relations   \eqref{eq: relation-lestrygon-geral-1},  \eqref{eq: relations H}, \eqref{eq: relation-lestrygon-dual1},  \eqref{eq relations H dual-op1} and
	\eqref{eq:double-relations-1}. The comultiplication is determined by those of $H_p(q,a)$  and $K_p(q,a)^{\operatorname{op}}$, given in \eqref{eq: Delta H} and \eqref{eq: Delta H dual}.

\smallbreak
\item\label{prop:Dtilde-dos} The monomials 
\begin{align}\label{eq:pbw-Dtilde}
	z_0^{n_0}\dots z_{\ell}^{n_{\ell}}\xb_2^{r_2}\xb_1^{r_1}\xb_{21}^{r_{21}}\alpha_1^{k_1} \alpha_2^{k_2} T_1^{k_3} T_2^{k_4} \zeta_1^{q_1} \zeta_2^{q_2}  \ub_{21}^{s_{21}}\ub_1^{s_1} \ub_2^{s_2} w_0^{m_0}\dots w_{\ell}^{m_{\ell  }},
\end{align}
form a PBW-basis of $\mathcal{DL}_p(q, a)$, where $r_1,s_1\in \I_{0,1}$, $k_i\in\Z$ for every $i$, and all other exponents belong to $\N_0$. In particular, $\GK\mathcal{DL}_p(q, a) = 2(\ell + 6)$. 
\smallbreak\item\label{prop:Dtilde-tres} There exists a Hopf algebra epimorphism $\mathcal{DL}_p(q,a)\twoheadrightarrow \mathtt{D_p(q,a)}$. 
\end{enumerate}
\end{prop}
\begin{proof}
We outline the argument which is similar to that of \cite[Proposition 2.3]{APenha}. \ref{prop:Dtilde-uno} follows from the definition.
For \ref{prop:Dtilde-dos}, let $A$ denote the algebra presented by the generators and relations in \ref{prop:Dtilde-uno}.  By construction, the relations are valid in $\mathcal{DL}_p(q, a)$ and thus we have an epimorphism	$A\twoheadrightarrow \mathcal{DL}_p(q, a)$. Using the commutation relations we see that $A$ is linearly generated by \eqref{eq:pbw-Dtilde} and then follows from \ref{prop:Dtilde-uno} that $A\simeq\mathcal{DL}_p(q, a)$.
Notice that	
	\ref{prop:Dtilde-tres} follows from \ref{prop:Dtilde-dos} since the relations in $\mathcal{DL}_p(q,a)$ are also satisfied in $\mathtt{D_p(q,a)}$. 
\end{proof}

\section{Extensions}\label{sec:extensions}
For the rest of the paper, we assume  that $q=1$ which implies $N=1$ and $T_i=1$, $i\in \I_2$. Let $M$ be the subalgebra of $\DLstr_p(1, a)$ generated by $x_{21}^p$, $u_{21}^p$, $x_i^{[p]}$, $\alpha_i^p$, $\ub_i^{[p]}$, $z_n^{p}$, $w_n^p$, $\zeta_i^{[p]}$, where $\zeta_i^{[p]} = \zeta_i^{p}- \zeta_i,\, u_i^{[p]}=u_i^{(1+\delta_{i,p})p}$ and $x_i^{[p]}=x_i^{(1+\delta_{i,p})p}$, $i\in\I_2$, $n\in \I_{\ell}$. It follows from Lemmas \ref{lemma:relations1}, \ref{lemma:relations2}, \ref{lemma:relations3}\, and
	\ref{lemma:comultiplication-powers-Dtilde} that $M$ is a normal commutative Hopf subalgebra of $\DLstr_p(1, a)$.  As an algebra, $M$ is isomorphic to the quotient of the polynomial algebra
	\[\ku [\mathtt{x}_i,\mathtt{u}_i,\mathtt{\xi}_i,\mathtt{a}_i^{\pm},\mathtt{z}_j,\mathtt{w}_j,\mathtt{y}_0,\mathtt{v}_0\mid {i\in \I_{3},\,  j\in \I_{0,\ell}}] \]
		 by the ideal generated by $\langle\mathtt{x}^2_1+2\mathtt{y}_{0}, \mathtt{u}^2_1+ 2\mathtt{v}_{0}\rangle$, via the identification
	\begin{align*}
		&x_i^{[p]}\mapsto	\mathtt{x}_i,& & \alpha_i^p \mapsto	\mathtt{a}_i, & & z_j^p\mapsto	\mathtt{z}_j, & & x_{21}^p\mapsto	\mathtt{y}_{0},\\
		&u_i^{[p]}\mapsto	\mathtt{u}_i,& & \zeta_i^{[p]}\mapsto	\mathtt{\xi}_i, & & w_j^p\mapsto	\mathtt{w}_j, & &u_{21}^p\mapsto	\mathtt{v}_{0}.
	\end{align*}
Notice that $M$ is a Noetherian domain and $M$ is isomorphic to $\mathcal{O}(\Gb)$ as a Hopf algebra, for the algebraic group $\Gb=\Hom_{\operatorname{Alg}}(M,\ku) $. \vspace{.1cm}

\subsection{An exact sequence}	
We recall from \cite{A-Devoto} and \cite{Sch} that a short
	exact sequence of Hopf algebras is a  sequence of Hopf algebra  morphisms
	$A\xhookrightarrow[]{\iota} C \xrightarrowdbl[]{\pi} B$ satisfying
	\begin{multicols}{2}
		\begin{enumerate}[leftmargin=*,label=\rm{(\roman*)}]
			\item\label{suc-exacta-1} $\iota$ is injective.
			\item\label{suc-exacta-2} $\pi$ is surjective.
			\item\label{suc-exacta-3} $\ker\pi = C\, (\iota A)^+$.
			\item\label{suc-exacta-4} $\iota(A) = C^{\operatorname{co} \pi}$.
		\end{enumerate}
	\end{multicols}

	\begin{remark}\label{remark-exact-sequence-hopf}
		If $A\xhookrightarrow[]{\iota} C$ is faithfully flat and $\iota(A)$ is stable under the left adjoint action of $C$ then
		\ref{suc-exacta-1}, \ref{suc-exacta-2} and \ref{suc-exacta-3} imply \ref{suc-exacta-4}, see  \cite[Corollary 1.2.5]{A-Devoto}.
	\end{remark}

	Let   $\mathbf{R}$ be the Hopf subalgebra of $\mathtt{D_p(1,a)}$ generated by $ \xb_1,\ub_1, \xb_{21}, \ub_{21}, \alpha_1,\alpha_2$. Since $q=1$, $\mathbf{R}$ is commutative. Our first goal is to describe  $\mathbf{R}$ as a quotient of the algebra of functions of an algebraic group.  Consider the affine group schemes
\[{\footnotesize \mathbb{\tilde{G}}_1=\left\lbrace{\tiny  \begin{pmatrix}
1&0&0\\
t_1&\lambda&0\\
t_2&- t_1 \lambda&\lambda^2
\end{pmatrix} } : \begin{array}{l}
t_1^2=-2t_2, \\ \lambda\neq 0.
\end{array}\right\rbrace,
\quad \mathbb{\tilde{G}}_2=\left\lbrace{\tiny  \begin{pmatrix}
1&0&0\\
s_1&1&0\\
s_2&-s_1&1
\end{pmatrix} } : \begin{array}{l} s_1^2=-2 s_2 \end{array} \right \rbrace   }. \]
Let $\mathbf{B} := \mathbb{\tilde{G}}_1\times \mathbb{\tilde{G}}_2\times \Gb_m$.  Notice that if $p\neq 2$, then $\mathcal O(\mathbf B)	\simeq
\Bbbk[t_1,s_1,\lambda^{\pm1},\mu^{\pm1}]$ and hence $\mathbf B$ is reduced. Assume that $p=2$. In this case, $t_1^2=s_1^2=0$ and consequently $\mathbf B$ is a nonreduced affine group scheme. The
nilradical of $\mathcal O(\mathbf B)$  is generated by $t_1$ and $s_1$. The reduced closed subgroup scheme of $\mathbf B$ is
\[
\mathbf B_{\mathrm{red}}
=
(\widetilde G_1)_{\mathrm{red}}
\times
(\widetilde G_2)_{\mathrm{red}}
\times\mathbb G_m,
\]
where
\[
(\widetilde G_1)_{\mathrm{red}}
=
\left\lbrace{\tiny 
\begin{pmatrix}
	1&0&0\\
	0&\lambda&0\\
	t_2&0&\lambda^2
\end{pmatrix}}
:\lambda\neq0,\,t_2\in \Bbbk
\right\rbrace ,
\,\,\,\,
(\widetilde G_2)_{\mathrm{red}}
=
\left\lbrace{\tiny 
\begin{pmatrix}
	1&0&0\\
	0&1&0\\
	s_2&0&1
\end{pmatrix} }: s_2\in \Bbbk
\right\rbrace .
\]
Consequently, $\mathcal O(\mathbf B_{\mathrm{red}})\simeq \Bbbk[t_2,s_2,\lambda^{\pm1},\mu^{\pm1}]$.

\begin{prop} The following assertions hold:
\begin{enumerate}[leftmargin=*,label=\rm{(\roman*)}]
		\item  $\mathcal{O}(\mathbf{B}) \simeq \ku[\mathtt{x}_1,\mathtt{y}_0,\mathtt{a}_1^{\pm},\mathtt{u}_1,\mathtt{v}_0,\mathtt{a}_2^{\pm}]/\langle\mathtt{x}_1^2+2\mathtt{y}_0,\mathtt{u}_1^2+2\mathtt{v}_0\rangle$ as algebras, \vspace{.1cm}

		\item There exists an exact sequence of Hopf algebras
				\begin{align*}
			\xymatrix{ \cO(\mathbf{B}_{\rm{red}})  \ar@{^{(}->}[rr] ^{\operatorname{Fr}} & & \cO(\mathbf{B}) \ar@{->>}[rr] ^{\phi} & & \nucleo },
		\end{align*}
		where $\operatorname{Fr}$ is the Frobenius morphism.
	\end{enumerate}
\end{prop}

\begin{proof}
The first assertion follows directly from the definition. For (ii), define the comultiplication on $\ku[\mathtt{x}_1,\mathtt{u}_1,\mathtt{y}_0,\mathtt{v}_0,\mathtt{a}_1^{\pm},\mathtt{a}_2^{\pm}]$  by
\begin{align*}
	\Delta(\mathtt{x}_1) &= \mathtt{x}_1\ot 1 +\mathtt{a}_1 \ot \mathtt{x}_1, & \Delta(\mathtt{u}_1) &= \mathtt{u}_1\ot 1 + 1 \ot \mathtt{u}_1,\\
	\Delta(\mathtt{y}_0) &= \mathtt{y}_0 \ot 1-\mathtt{x}_1 \mathtt{a}_1 \ot \mathtt{x}_1+\mathtt{a}_1^2\ot \mathtt{y}_0 ,& \Delta(\mathtt{a}_1) &= \mathtt{a}_1 \ot \mathtt{a}_1,\\
	\Delta(\mathtt{v}_0) &= \mathtt{v}_0 \ot 1  -\, \mathtt{u}_1 \ot \mathtt{u}_1+1\ot \mathtt{v}_0 ,& \Delta(\mathtt{a}_2) &= \mathtt{a}_2 \ot \mathtt{a}_2.
\end{align*}
Since 
\begin{align*}
	\Delta(\mathtt{x}_1^2+2\mathtt{y}_0) &= (\mathtt{x}_1^2+2\mathtt{y}_0)\ot 1+\mathtt{a}_1^2\ot (\mathtt{x}_1^2+2\mathtt{y}_0),\\
	\Delta(\mathtt{u}_1^2+2\mathtt{v}_0) &= (\mathtt{u}_1^2+2\mathtt{v}_0)\ot 1+1\ot (\mathtt{u}_1^2+2\mathtt{v}_0),
\end{align*}
then $I=\langle\mathtt{x}_1^2+2\mathtt{y}_0,\mathtt{u}_1^2+2\mathtt{v}_0\rangle$ is a Hopf ideal. Hence $\cO(\mathbf{B})$ is a Hopf algebra. Thus the map 
$\operatorname{\operatorname{Fr}}: \cO(\mathbf{B}_{\rm{red}})\to \cO(\mathbf{B})$ given by $\operatorname{Fr}(b)=b^p,$\,  for $b\in\cO(\mathbf{B}_{\rm{red}})$ is a Hopf algebra monomorphism. Let $\phi\colon \cO(\mathbf{B}) \longrightarrow \mathbf{R}$ defined by
\begin{align*}
	\phi(\mathtt{x}_1) &= \xb_1,&\phi(\mathtt{y}_0) &= \xb_{21},& \phi(\mathtt{a}_1) &= \alpha_1,\\
	\phi(\mathtt{u}_1) &= \ub_1,& \phi(\mathtt{v}_0) &=\ub_{21},& \phi(\mathtt{a}_2) &= \alpha_2,
\end{align*}
which is a Hopf algebra epimorphism. The morphism $\operatorname{Fr}$ induces a structure of $\cO(\mathbf{B}_{\rm{red}})$-module on $\cO(\mathbf{B})$. With respect to this structure, $\cO(\mathbf{B})$ has a finite basis, given by the set
\begin{align*}
	\{ 	\mathtt{x}_1^{k_1}\mathtt{u}_1^{k_2}\mathtt{y}_0^{n_1}\mathtt{v}_0^{n_2}\mathtt{a}_1^{\pm m_1}\mathtt{a}_2^{\pm m_2}\colon k_1,k_2\in \I_{0, 1} ,n_1,n_2\in \I_{0,\wp},m_1,m_2\in \I_{0, p-1}\},
\end{align*}
where $\wp:= \frac{p-1}{2-\delta_{2,p}}$. Thus, $\cO(\mathbf{B})$ is flat. Since $\operatorname{Fr}(\cO(\mathbf{B}_{\rm{red}}))$ is adjoint invariant,  Remark \ref{remark-exact-sequence-hopf} applies. Hence we only need to check that $\ker \phi = \cO(\mathbf{B})\operatorname{Fr}(\cO(\mathbf{B}_{\rm{red}}))^+$. For this we sketch the main idea of the computations. Let $\mathtt{b}_{i,\pm} = \mathtt{a}_i^{\pm p} - 1$. Then 
\begin{align*}
	\mathtt{S} =\{\mathtt{x}_1^{k_1} \mathtt{u}_1^{k_2}\mathtt{y}_0^{n_1}\mathtt{v}_0^{n_2} \mathtt{a}_1^{\pm m_1} \mathtt{b}_{1,\pm}^{n_3} \mathtt{a}_2^{\pm m_2} \mathtt{b}_{2,\pm}^{n_4}\colon n_i\in\N_0, m_j\in \I_{0, p-1}, k_i\in \I_{0, 1}\}
\end{align*}
generates $\cO(\mathbf{B})$ as a vector space. Using the set of generators $\mathtt{S}$, the computations are straightforward.
\end{proof}

 \subsection{A restricted enveloping algebra quotient}\label{sec:enveloping}

We recall the definition of the Lie algebra $\mathfrak{fsl}_2$, called the fake $\mathfrak{sl}(2)$ Lie algebra. This is the three-dimensional Lie algebra with basis $\{e_{o},h_{o},f_{o}\}$ and brackets
\begin{align*}
& [h_{o},e_{o}]=e_{o}, & & [h_{o},f_{o}]=-f_{o}, & & [e_{o},f_{o}]=h_{o}.
\end{align*}
This algebra can be realized as the Lie subalgebra of $\mathfrak{sl}_3$ generated by
$\{E_{12}-E_{23},\,E_{11}-E_{33},\,E_{21}-E_{32}\}$, where $E_{ij}$ denotes the matrix with $1$ in the $(i,j)$-entry
and $0$ elsewhere. If $p\neq 2$, then $\mathfrak{sl}_2\simeq \mathfrak{fsl}_2$ via the isomorphism 
\begin{align}\label{eq: iso fsl2-sl2}
& e\mapsto 2e_{o}, & & f\mapsto f_{o}, & & h\mapsto 2 h_{o}.
\end{align}
Let $\mathfrak{m}$ denote the minimal $p$-envelope of $\mathfrak{fsl}_2$. If $p\neq 2$, then $\mathfrak{m}\simeq \mathfrak{fsl}_2\simeq \mathfrak{sl}_2$. On the other hand,  when  $p=2$, $\mathfrak{m}$ is a $5$-dimensional Lie algebra with basis $\{e_o,\tilde{e}_{o},f_o,\tilde{f}_{o},h_o \}$ (see for instance \cite[$\S$ 2]{A-Bagio-FF3}). Besides the defining brackets of $\mathfrak{fsl}_2$, its additional nonzero brackets are
\begin{align*}
[\tilde{e}_{o},\tilde{f}_{o}] = h_o,\qquad [\tilde{e}_{o},f_o] = e_o,\qquad [e_o,\tilde{f}_{o}] = f_o.
\end{align*}
To unify the notation, for $p\neq 2$ we define $\tilde{e}_{o}=\tilde{f}_{o}=0$.

Let $\mathfrak h_0=\mathfrak{fsl}_2\oplus \ku\{c\}$ be the Lie algebra with one dimensional center $\ku\{c\}$. Let $\mathfrak{a}_{\ell} = \ku \{s_0,\dots,s_{\ell},r_0,\dots,r_{\ell}\}$ be a $2(\ell+1)$-dimensional abelian Lie algebra. Setting $\beta_{n,a}:= n a + \binom{n}{2}$,  we define the action of $\mathfrak{h}_0$ on $\mathfrak{a}_{\ell}$ by
\begin{align*}
& e_{o} \rightharpoonup s_n = s_{n+1}, & & f_{o}\rightharpoonup s_n = -\beta_{n,a} s_{n-1}, & & h_{o}\rightharpoonup s_n = (n+a) s_n,\\
& e_{o}\rightharpoonup r_n = \beta_{n,a} r_{n-1}, & & f_{o} \rightharpoonup r_n = -r_{n+1},   & & h_{o}\rightharpoonup r_n = -(n+a) r_n,\\
& c\rightharpoonup s_n = s_n, & & c\rightharpoonup r_n = -r_n, & & n\in \I_{0, \ell},
\end{align*}
with the convention $s_j=r_j=0$ whenever
$j\notin\I_{0,\ell}$. For the case $p=2$, the previous action extends to the minimal $p$-envelope
$\mathfrak{h}:=\mathfrak{m}\oplus \ku\{c\}$ of $\mathfrak h_0$  by requiring it to be restricted. Thus,
\begin{align*}
&	\tilde{e}_o \rightharpoonup s_n =s_{n+2}, & & \tilde{f}_o \rightharpoonup s_n =\beta_{n-1,a}\ \beta_{n,a}\, s_{n-2}, \\
&  \tilde{e}_o \rightharpoonup r_n
=\beta_{n-1,a}\ \beta_{n,a}\, r_{n-2},	& &
\tilde{f}_o \rightharpoonup r_n =r_{n+2}.
\end{align*}
The Lie algebra $\g \coloneqq \mathfrak{a}_{\ell} \rtimes \mathfrak{h}$ is restricted with $p$-operation given by 
\begin{align*}
& e_o^{[p]} = \tilde{e}_{o},\,  & & f_o^{[p]} = \tilde{f}_{o},& &  \tilde{e}^{[p]}_o =\tilde{f}^{[p]}_o = 0,\\
& s_n^{[p]} = r_n^{[p]} = 0,& & h_o^{[p]} = h_o,&  & c^{[p]} = c.
\end{align*}

\FloatBarrier		
\begin{prop}\label{Prop:exact-sequence-double-fin-dim}
There exists an abelian exact sequence of Hopf algebras		
\begin{align*}
\xymatrix{ \nucleo   \ar@{^{(}->}[rr] ^{\iota} & & \mathtt{D_p(1,a)} \ar@{->>}[rr] ^{\pi} & & \mathfrak{u}(\g), }
\end{align*}
where $\mathfrak{u}(\g) $ is the universal restricted enveloping algebra of $\g= \mathfrak{a}_{\ell} \rtimes \mathfrak{h}$.
\end{prop}

\begin{proof}
Consider the natural inclusion $\iota : \nucleo \longrightarrow  \mathtt{D_p(1,a)}$. A direct computation, using Table \ref{table: adjoint R} below, shows that $\nucleo$ is a normal Hopf subalgebra of $\mathtt{D_p(1,a)}$.
\begin{table}[H]	
\begin{tabular}{|c|c|c|c|c|c|c|}
\hline
\diagbox{h}{y}  & $x_1$ 	& $x_{21}$ & $u_1$  & $u_{21}$  & $\alpha_1$  & $\alpha_2$    \\\hline
$z_0$		&  $0$			&  $0$  & $0$ 	& $0$ & $0$ & $0$  \\\hline
$x_2$		& 	$x_{21}$  	&  $ -x_1x_{21}$  & $ \alpha_1-1$ &  $-(\alpha_1-1)\ub_1$ & $-\alpha_1\xb_1$ & $ \,    -a  \xb_1 \alpha_2$   \\\hline
$\zeta_1$	& $x_1$ 		& $2\xb_{21}$  & $-u_1$ & $-2\ub_{21}$  & 0 & 0 \\\hline
$\zeta_2$	& $0$ 			&  $0$  & $0$& $ 0 $ & 0 & 0 \\\hline
$u_2$		& $-(\alpha_1-1)$ 	&$\xb_1(\alpha_1-1) $  & $-u_{21}$	& $u_{21}u_{1}$  &  $\alpha_1 u_1$ & $a\,  \alpha_2 u_1$ \\\hline
$w_0$		& $0$  			& $0$  &  $0$& $0$  & $0$ & $0$ \\\hline
\end{tabular}
\caption{Left adjoint action $h\cdot y$ of $\mathtt{D_p(1,a)}$ on $\nucleo$}
\label{table: adjoint R}
\end{table}
Let $\pi: \mathtt{D_p(1,a)}\longrightarrow \mathfrak{u}(\g)$ be the map given by 
\begin{align}\label{eq: hom D-pi-u(g)}
\begin{aligned}
&	\pi(z_i)=s_i, & & \pi(w_i)= r_i, & & \pi(\zeta_2)=c,\\
&	\pi(x_2)=e_{o}, & & \pi(\zeta(a))=h_{o}, & & \pi(u_2)= f_{o}.
\end{aligned}
\end{align}
Clearly $\pi$ is a Hopf algebra epimorphism  with kernel $\mathtt{D_p(1,a)}\iota(\nucleo)^+$. The result follows from Remark \ref{remark-exact-sequence-hopf}
\end{proof}

\subsection{A commutative square of Hopf algebras}

Let $\widetilde{\nucleo}$ be the Hopf subalgebra of $\DLstr_p(1, a)$ generated by $ x_1,u_1,\xb_{21}, \ub_{21},\alpha_1,\alpha_2$. From Proposition \ref{prop:Dtilde} \ref{prop:Dtilde-tres} there exists a Hopf algebra epimorphism $\DLstr_p(1,a)\twoheadrightarrow \mathtt{D_p(1,a)}$.  It is clear that $\nucleo$ be the image of $\widetilde{\nucleo} $ under  this epimorphism. Denote $\g_0:=\mathfrak{a}_{\ell} \rtimes \mathfrak h_0=\mathfrak{a}_{\ell} \rtimes (\mathfrak{fsl}_2\oplus \ku\{c\})$. If $p=2$, then $\g_0$  is not restricted and has 2-envelope $\g=\mathfrak{a}_{\ell} \rtimes \mathfrak{h}$ (see the previous subsection).

\begin{prop}
	Every row and column of the following diagram is an exact sequence of Hopf algebras.
	\begin{equation}\label{eq:cuadrado_magico}
		\begin{aligned}
			&\xymatrix{ \mathcal{O}(\mathbf{B}_{\rm{red}})  \ar@{^{(}->}[d] ^{\rm{Fr}} \ar@{^{(}->}[rr] ^{{\iota_{\Gb}}} & & \mathcal{O}(\Gb)  \ar@{->>}[rr] ^{\pi_{\Gb}}\ar@{^{(}->}[d] ^{} & & \mathcal{O}(\Gb_a^{ 2(\ell+3)})\ar@{^{(}->}[d] ^{}\\
				\mathcal{O}(\mathbf{B})\ar@{->>}[d] ^{}   \ar@{^{(}->}[rr] ^{\widetilde{\iota}} & & \DLstr_p(1, a)  \ar@{->>}[rr] ^{\pi_{\DLstr}}   \ar@{->>}[d] ^{}& & U(\g_0)  \ar@{->>}[d] ^{\psi}\\
				\nucleo   \ar@{^{(}->}[rr] ^{\iota} & & \mathtt{D_p(1,a)} \ar@{->>}[rr] ^{\pi} & & \mathfrak{u}(\g)}
		\end{aligned}
	\end{equation}
\end{prop}

\begin{proof}
Consider the Hopf algebra epimorphism $\psi:U(\mathfrak g_0)\longrightarrow\mathfrak u(\mathfrak g)$ defined by
\begin{align*}
& s_i\mapsto s_i, & &  r_i\mapsto r_i, & &
e_o\mapsto e_o, & &  f_o\mapsto f_o, & & 
h_o\mapsto h_o, & &  c\mapsto c.
\end{align*}
For $p=2$, $\psi$ is still surjective since $\widetilde e_o=e_o^2=\psi(e_o^2)$ and $\widetilde f_o=f_o^2=\psi(f_o^2)$. 
Let $Z$ be the central Hopf subalgebra of $U(\mathfrak g_0)$
generated by
\begin{align*}
& s_i^p, & & r_i^p, & &  h_o^p-h_o, & & c^p-c, & & 
e_o^{p(1+\delta_{2,p})}, & & 
f_o^{p(1+\delta_{2,p})}.
\end{align*}
Then $Z\simeq\mathcal O(\mathbf G_a^{2(\ell+3)})$ and $\ker\psi=U(\mathfrak g_0)Z^+$. Consequently, the right-hand column is an exact sequence of Hopf
algebras. Also, consider the Hopf algebra morphism $\pi_{\DLstr}: \DLstr_p(1,a)\longrightarrow U(\mathfrak g_0)$ defined by
\begin{align*}\begin{aligned}
&	z_i \longmapsto s_i,& & w_i\longmapsto r_i, & & x_2\longmapsto e_o, & & u_2\longmapsto f_o,\\
& \zeta(a)\longmapsto h_o,& & \zeta_2\longmapsto c, & & \alpha_1,\alpha_2\longmapsto1,&
&x_1,x_{21},u_1,u_{21}\longmapsto0.
\end{aligned}\end{align*}
A direct verification using the defining relations and comultiplications shows that $\pi_{\DLstr}$ is a well-defined Hopf algebra
epimorphism and $\ker\pi_{\DLstr}=\DLstr_p(1,a)\, \widetilde\iota\bigl(\mathcal O(\mathbf B)^+\bigr)$.
The restriction of $\pi_{\DLstr}$ to $\mathcal O(\mathbf G)$
induces the Hopf algebra epimorphism $\pi_{\mathbf G}$ given by
\begin{align*}\begin{aligned}
		& z_i^p\longmapsto s_i^p,&
		& w_i^p\longmapsto r_i^p, &
		& x_2^{p(1+\delta_{2,p})}
		\longmapsto e_o^{p(1+\delta_{2,p})},\\
		&u_2^{p(1+\delta_{2,p})}
		\longmapsto f_o^{p(1+\delta_{2,p})}, &
		& \zeta(a)^p-\zeta(a)
		\longmapsto h_o^p-h_o,&
		& \zeta_2^p-\zeta_2
		\longmapsto c^p-c.
	\end{aligned}
\end{align*}
Moreover, $\ker\pi_{\mathbf G}=\mathcal O(\mathbf G)\,\iota_{\mathbf G} \bigl(\mathcal O(\mathbf B_{\mathrm{red}})^+\bigr)$ and hence the upper row is an exact sequence of Hopf algebras. The exactness of the remaining rows and columns follows from the previous results, while the commutativity of the diagram follows
directly from the definitions of the maps.
\end{proof}

\section{Simple modules of \texorpdfstring{$\mathtt{D_p(1,a)}$}.} \label{section:simplemodules}

Throughout this section, we use the notation introduced in Section~\ref{sec:extensions}. In particular, we write $\mathfrak h=\mathfrak m\oplus\ku\{c\}$ and $\mathfrak g=\mathfrak a_\ell\rtimes\mathfrak h$. The following lemma from \cite{A-Dumas-Penha} allows us to determine the finite-dimensional simple modules of $\mathtt{D_p(1,a)}$.
\begin{lemma}\label{lema:quotient-nilpotent} \cite[Lemma 1.11]{A-Dumas-Penha}
	Let $A$ be an algebra and $\cF\subseteq A$ be a family of elements satisfying
	\begin{enumerate}[leftmargin=*,label=\rm{(\alph*)}]
		\item\label{item:quotient-nilpotent-commuting} the elements of $\cF$ commute with each other;
		
		\item\label{item:quotient-nilpotent} for any $M\in\lmod{A}$, with $\dim M<\infty$, any $x\in\cF$ acts nilpotently on $M$;
		
		\item\label{item:quotient-nilpotent-normal} $\cF$ is normal, i.e., the vector subspace $I$ of $A$ generated by $\cF$ satisfies
		\begin{align}\label{eq:left-twosided}
			AI=IA.
		\end{align}
	\end{enumerate}
Let $S$ be a  finite-dimensional simple left $A$-module. Then $S$ is a simple left $A/IA$-module.
	Thus the projection $A\to A/IA$ induces a bijection
	\begin{align*}
		\irr A\simeq\irr A/IA.
	\end{align*}
\end{lemma}

To describe the finite-dimensional simple $\mathtt{D_p(1,a)}$-modules, given $0\leq k<p$, we take
\begin{equation}
	d(k):=
	\begin{cases}
		2k+1, & \text{if }0\leq k\leq\tfrac{p}{2},\\
		2k+1-p, & \text{if }\tfrac{p}{2}<k<p.
	\end{cases}
\end{equation}
For each $k,\mu\in\I_{0,p-1}$, consider the $\mathtt{D_p(1,a)}$-module $V_{k,\mu}$ with basis $\{v_0,\ldots,v_{d(k)-1}\}$ and action given by: for all $i\in \I_{0,d(k)-1}$,
\begin{align}\label{eq: rep simples D(H)}
	\begin{aligned}
		& u_2\cdot v_i=v_{i+1},& &
		\zeta(a)\cdot v_i=(k-i)v_i,& &
		x_2\cdot v_i=\left(ik-\tbinom{i}{2}\right)v_{i-1},\\
	&	\zeta_2\cdot v_i=\mu\,v_i,& &
		\alpha_j\cdot v_i=v_i,& &
		w_n\cdot v_i=0,\\
	&	z_n\cdot v_i=0,& &
		u_1\cdot v_i=0,& &
		x_1\cdot v_i=0,
	\end{aligned}
\end{align}
with the convention that $v_{-1}=v_{d(k)}=0$.

\begin{theorem}\label{simple-modules}
The complete set of representatives of the isomorphism classes of finite-dimensional simple $\mathtt{D_p(1,a)}$-modules is
\[\irr \mathtt{D_p(1,a)}= \{V_{k,\mu}:k,\mu\in\mathbb I_{0,p-1}\}. \]
\end{theorem}

\pf
Taking $A=\mathtt{D_p(1,a)}$ and $\cF_1=\{\xb_1,\ub_1,\xb_{21},\ub_{21},\alpha_1-1,\alpha_2-1\}$, the commutation relations, Proposition \ref{Prop:exact-sequence-double-fin-dim} and Lemma \ref{lema:quotient-nilpotent} yield $\irr\mathtt{D_p(1,a)}\simeq\irr\mathfrak{u}(\g)$. Again, we can apply Lemma \ref{lema:quotient-nilpotent} with $A=\mathfrak{u}(\g)$ and $\cF_2=\{s_n,r_n:n\in\I_{0,\ell}\}$ and we
conclude that $\irr\mathfrak{u}(\g)\simeq\irr\mathfrak{u}(\mathfrak h)$.
The finite-dimensional simple $\mathfrak{u}(\mathfrak h)$-modules are well known. Indeed, given $k,\mu\in\I_{0,p-1}$, let $V_{k,\mu}$ be the vector space with basis $\{v_0,\ldots,v_{d(k)-1}\}$, with the action of $\mathfrak{u}(\mathfrak h)$ on $V_{k,\mu}$ given by: for all $i\in \I_{0,d(k)-1}$,
\begin{align}\label{eq: rep simples u(g)}
	\begin{aligned}
		& e_o\cdot v_i=\left(ik-\tbinom{i}{2}\right)v_{i-1}, &
		& f_o\cdot v_i=v_{i+1},\\
		& h_o\cdot v_i=(k-i)v_i, &
		& c\cdot v_i=\mu v_i,
	\end{aligned}
\end{align}
with the convention that $v_{-1}=v_{d(k)}=0$. For $p=2$, we have the additional actions 

\begin{align*}
& \widetilde e_o\cdot v_i=e_o^2\cdot v_i, & &  \widetilde f_o\cdot v_i=f_o^2\cdot v_i.
\end{align*}
Observe that for $p\neq2$, the formulas \eqref{eq: rep simples u(g)} are simply the rephrasing of the irreducible representations of $\mathfrak{u}(\mathfrak{sl}_2)\simeq 
\mathfrak{u}(\mathfrak h)$ under the isomorphism \eqref{eq: iso fsl2-sl2}. On the other hand, for $p=2$, the formulas \eqref{eq: rep simples u(g)} recover the result \cite[Proposition 3.1]{A-Bagio-FF3}.
Finally, the formulas in \eqref{eq: rep simples D(H)} follow from \eqref{eq: rep simples u(g)} and the epimorphism $\pi$ given in \eqref{eq: hom D-pi-u(g)}, 
by letting all elements of \(\cF_1\cup\cF_2\) act trivially. Therefore, the modules $V_{k,\mu}$, with
$k,\mu\in\mathbb I_{0,p-1}$, form a complete set of pairwise non-isomorphic finite-dimensional simple $\mathfrak u(\mathfrak h)$-modules.
\epf

\appendix


\section{Some formulas in \texorpdfstring{$\DLstr(q, a)$}.} \label{sec:formulas}
In this section, we collect some commutation formulas used in $\S$ \ref{sec:extensions}.  Recall that the raising factorial is defined by
\begin{align*}
[X]^{[0]}=1,&&	[X]^{[n]} &= \prod_{i=1}^{n} (X+i-1)\in \Z[X], \qquad n\geq 1.
\end{align*}
In particular, note that if $\car\ku=p\neq 0$ then $[X]^{[p]} = X^p-X$. To state the forthcoming formulas more concisely, we introduce some
notation. Let $s\in\Bbbk$ and $t\in \N_0$ and $y\in \{x,u\}$. We set
\begin{equation}\label{eq: NotationShort} 
	y_{\{j:s\}}:=y_j+s\, \delta_{2,j}y_1,\qquad \text{and } \qquad y^t_{\{j:s\}}:=\left(y_{\{j:s\}}\right)^t.
\end{equation}
Also, for  $n,m\in \N_0,$ we define 
\begin{equation}\label{eq: NotationLambdaCoefs} 
	\begin{aligned}
		\beta_{n,X}&:= n X+\tbinom{n}{2}\in \Z[X],
\qquad 		\Upsilon_{n,m}&:=\begin{cases}
			u_1^{m-n} &\text{ if } n<m,\\
			1 -\alpha_2T_1 (\alpha_1T_2)^m&\text{ if } n=m,\\
			-x_1^{n-m}\alpha_2T_1 (\alpha_1T_2)^m &\text{ if } n>m.
		\end{cases}
	\end{aligned}
\end{equation}
We note that $\beta_{n,X}$ was already defined in Section \ref{sec:enveloping}; we recall it here for the convenience of the reader.

\subsection{Case $p\neq 2$}
Suppose that $\car\ku = p\neq 2$.

\begin{lemma}\label{lemma:relations1} Let $r,t\in \N, s\in \Bbbk,$ and $H_{p}(q,a)$ be the algebra defined in \eqref{eq:def-H}. The following formulas hold in $H_{p}(q,a)$:
	\begin{align*}
		x_{21}^r x^t_{\{2:s\}} &=x^t_{\{2:s+r\}}\,  \xb_{21}^r,& z_m z_n &= q^{m-n} z_n z_m,\\
		\alpha_1^{r} x_{\{2:s\}}^{t}&=  x^t_{\{2:s+r\}}  \ \alpha_1^{r},& \alpha_2^{r} x_{\{2:s\}}^{t}&=  q^{-r t}   x^t_{\{2:s+  a r\}}  \alpha_2^{r},\\
		x_2 z_n^{t} &= q^{t} z_n^{t} x_2+ t\, z_{n+1} \, z_{n}^{t-1},& x_2^{r} z_n&= \sum_{i= 0}^{r}  \binom{r}{i}  q^{r-i}\,\, z_{n+i} \, x_2^{r-i},\\
		x_1^{r} x^{t}_{\{2:s\}}&=  \xb^t_{\{2:\tiny{s+\sfrac{r}{2} }\}}  x_1^{r}, & \xb^t_{\{2:s\}} &=\sum_{i=0}^t  \binom{t}{i} \tfrac{[2s]^{[i]}}{2^{i}} x_2^{t-i} x_1^{i},\\
	(s+t) x^{t}_{\{2:s/2\}}&=t x^{t-1}_{\{2:s/2\}}x_{\{2:s\}}+s x^{t}_{\{2:(s-1)/2\}}.\\
	\end{align*}
\end{lemma}

\begin{lemma}\label{lemma:relations2} Let $r,t\in \N,$ $s\in \Bbbk$,  $m,n\in \I_{0, \ell}$,  $j\in \I_2$ and $K(q,a)$ be the algebra defined in \eqref{alg-aux-cover}.  Using the notation \eqref{eq: NotationShort} and \eqref{eq: NotationLambdaCoefs}, the following formulas hold in $K(q,a)^{\operatorname{op}}$:
	\begin{align*}
		u_j^{r}\zeta_1^{t} &=  (\zeta_1+r )^{t} u_j^{r},&
		w_m w_n &= q^{m-n} w_n w_m,\\
		w_n^{r} u_2&=   q^{-r} u_2 w_n^{r}+ r w_n^{r-1}w_{n+1},&	
		w_n u_2^{t}&=   \sum_{i=0}^{t} \tbinom{t}{i} q^{-(t-i)}  u_2^{t-i} w_{n+i},  \\
		w_n^{r} \zeta_2^{t}&= (\zeta_2 +r)^{t} w_n^{r},& 
		w_n^{r} \zeta_1^{t}&= (\zeta_1+r n)^{t} w_n^{r},\\
		u_2^{r} u_1^{t} &=  u_1^{t} u^r_{\{2:\tiny{-\sfrac{t}{2} }\}}, &
		u^r_{\{2:s \}}&=\sum_{i=0}^r (-1)^i \binom{r}{i} \tfrac{[-2s]^{[i]}}{2^{i}} u_1^{i} u_2^{r-i}, \\
	\end{align*}
	\qedhere
\end{lemma}

\begin{lemma}\label{lemma:relations3} Let $r,t\in \N$ and $\DLstr(q, a)$ be the algebra defined in \eqref{alg-cover-positive}.  Using the notations \eqref{eq: NotationShort} and \eqref{eq: NotationLambdaCoefs}, the following formulas hold in $\DLstr(q, a)$:
\begin{align*}
u_2^{r} \alpha_1^{t} &= \alpha_1^{t}u^r_{\{2:t\}},\qquad\qquad\qquad\qquad\qquad\,\,\,  u_2^{r} \alpha_2^{t}=  q^{- r t}\alpha_2^{t}u^r_{\{2:at\}}, \\[.2em]
\zeta_1^{r} x_1^{t}&= x_1^{t}(\zeta_1+t )^{r},  \qquad\qquad \quad\qquad\qquad\,\,	\zeta_1^{r} x_2^{t} = x_2^{t}(\zeta_1+t )^{r},  \\[.2em]
\zeta_1^{r} z_n^{t} &=  z_n^{t}(\zeta_1+n t )^{r},\qquad\qquad\qquad\qquad\quad\! \zeta_2^{r} z_n^{t}=  z_n^{t}(\zeta_2+ t )^{r},\\[.2em]		
w_n^{r}  \xb_2&=  q^{r} \xb_2 w_n^{r} -rq^{r-1}\beta_{n,a}\alpha_1 T_2 w_{n-1}w_{n}^{r-1},\\[.7em]
w_n  x_2^{t}&=  \sum_{i=0}^{\min\{t,n\}}\, (-1)^i q^{t-i}\left( \textstyle\prod\limits_{k=0}^{i-1} \beta_{n-k,a} \right) \tbinom{t}{i} \,  x^{t-i}_{\{2:\tiny{\sfrac{i}{2} }\}}  (\alpha_1 T_2)^{i} w_{n-i},\\[.2em]
\ub_2^{r} z_n&=  \sum_{i=0}^{\min\{r,n\}}\, (-1)^i q^{-(r-i)}\left( \textstyle\prod\limits_{k=0}^{i-1} \beta_{n-k,a} \right) \tbinom{r}{i} \, \, z_{n-i}\,   \ub^{r-i}_{\{2:\tiny{a+n-\sfrac{i}{2} }\}}, \\[.2em]
\ub_2 z_n^{t}&=q^{-t} z_n^{t}  u_{\{2:t(a+n) \}} -q^{-(t-1)} t\,  \beta_{n,a}\, z_{n}^{t-1}z_{n-1},\\[.4em]
\ub_2  x_1^{t}&= \xb_1^{t} \ub_{\{2:t\}}  + t \xb_1^{t-1}(1-\alpha_1 ),\\
\ub_2 x_2^t & =x_2^t \ub_{\{2:t\}}+\tbinom{t}{2}x_2^{t-1}-t   \xb^{t-1}_{\{2:\tiny{\sfrac{1}{2} }\}}\alpha_1 T_2(t-1+\zeta(a)),\\
\ub_1 x_2^t & =x_2^t u_1+t \xb_2^{t-1}-  t  \xb^{t-1}_{\{2:\tiny{\sfrac{1}{2} }\}} \alpha_1 T_2,\\
\ub_j^r x_i &= x_i  \ub^r_{\{j:1\}} 	
+\delta_{i,1}\delta_{j,2}\,  \tbinom{r}{2} \left((r-2) \ub_{21} \ub^{r-3}_2 - \ub_{1} \ub^{r-2}_2 \right),\\
&\quad	+ \left(  r (1-\delta_{i,j}) (1-\alpha_1T_2)-\delta_{i,2}\delta_{j,2} \beta_{r,\zeta(a)}\,   \alpha_1T_2 \right)   \ub^{r-1}_{\{j:1\}}.\\
\end{align*}
Moreover, for  \(m,n\in\mathbb I_{0,\ell}.\), setting $S:=\max\{n,m\}$ and $s:=\min\{n,m\}$; then  we have
	\begin{align*}
		w_n^{r} z_m&= q^{r(m-n)}z_m w_n^{r}+ (-1)^{m} q^{r(m-n)}\, r \left( \textstyle\prod\limits_{k=1}^{S-s}\beta_{s+k,a}\right)   \tfrac{[2a]^{[S -s] }}{2^{S - s}}\, \,  \Upsilon_{n,m}\,  w_n^{r-1}, \\[.4em]
		w_n z_m^{t}&= q^{t(m-n)}z_m^{t} w_n+ (-1)^{m} q^{t(m-n)}\, t \left( \textstyle\prod\limits_{k=1}^{S-s}\beta_{s+k,a}\right)   \tfrac{[2a]^{[S -s] }}{2^{S - s}}\, \,   \, z_m^{t-1}\   \Upsilon_{n,m},  
	\end{align*}
	
\end{lemma}

\subsection{Comultiplication formulas }	
We begin with the formulas for the comultiplication of $z_n$ and $w_n$. Let $0\leq i\leq n$, and write uniquely $n-i=2k+j$, where $k\in\N_0$ and $j\in\{0,1\}$. Define
\begin{equation}\label{eq:Fni-Gni}
	\begin{aligned}
		F_{n,i} &:=   (-1)^{k}\tbinom{n}{i}\tfrac{(2(a+k)+i)^j}{2^j}\,\ \prod\limits_{l=0}^{k-1}  \tbinom{2(a+l)+i+1}{2}.
	\end{aligned}
\end{equation}

\begin{lemma}\label{lemma:comultiplication-zn-wn-Dtilde}
	The following comultiplication formulas are valid in $\DLstr(q, a)$:
	\begin{align*}
		\Delta(z_n)& =z_n\otimes 1 +\sum_{\substack{i+j+2k=n\\0\le j\le 1}}  (-1)^{j} F_{n,i}
		\,\,  x_1^j\xb_{21}^{k} \alpha_1^i \alpha_2 \otimes z_i,\\		
		\Delta(w_n)&=T_2^nT_1 \ot w_n +\sum_{\substack{i+j+2k=n\\ j\in\{0,1\}}} q^{-(n-i)} F_{n,i}\,  T_2^{n-i} w_i\ot \ub_{21}^k \ub_1^{j}.
	\end{align*}
\end{lemma}
\pf
The computations are done directly by induction, using the fact that the coefficients $F_{n,i}$ satisfy the recurrence: 
\begin{align*}
	F_{0,0}&=1,& 	F_{n+1,i} &=   (-1)^{n-i}((j+1)(a+k+i) +j)  F_{n,i}   +F_{n,i-1},
\end{align*}
where $n-i = 2k +j$.\epf

Now we give the formulas for the comultiplication of powers of the generators in $\DLstr(q, a)$. We only write the formulas that are non trivial.

\begin{lemma}\label{lemma:comultiplication-powers-Dtilde}
	For every $k\in\mathbb N_0$, the following comultiplication formulas
	hold in $\DLstr(q,a)$:
	\begin{align*}
		\Delta(x_2^k)
		&=
		\sum_{i=0}^k
		\binom{k}{i}
		x_{\{2:i/2\}}^{k-i}\alpha_1^i
		\otimes x_2^i,
		\\
		\Delta(u_2^k)
		&=
		\sum_{\substack{\mathbf{k}\in\mathbb N_0^4\\|\mathbf{k}|=k}}
		c_{\mathbf{k}}\,
		u_2^{k_1}
		[-2(\zeta_1+a\zeta_2)]^{[k_3]}
		T_2^{k-k_1}
		\otimes
		u_1^{k_2+k_3}u_2^{k_4},
	\end{align*}
	where, for
	$\mathbf{k}=(k_1,k_2,k_3,k_4)\in\mathbb N_0^4$, we set
	\[
	|\mathbf{k}|:=k_1+k_2+k_3+k_4,
	\qquad
	c_{\mathbf{k}}
	:=
	\binom{k}{k_1,k_2,k_3,k_4}
	\frac{(-1)^{k_3}}{2^{k_2+k_3}}
	[k_1]^{[k_2]}.
	\]
\end{lemma}
\pf
The computations are done directly by induction, using the last identity of lemma \ref{lemma:relations1}. The formula for $u_2$ can be obtained from the formula for $x_2$ applying Lemma \ref{lema: isomorfismo-V-inverse} and  \cite[Lemma 1.11]{A-Heckenberger-Schneider}.
\epf

\begin{coro}\label{cor:comultiplication-powers-p-Dtilde2}
	If $p>2$,	then:
	\begin{align*}
		\Delta(x_2^p) &=  \xb^{p}_{2} \ot 1+  \alpha_1^{p}\ot \xb_2^{p},\qquad\qquad \Delta(\ub_2^p)=  \ub^{p}_{2} \ot 1+  \alpha_1^{p}\ot \ub_2^{p} ,\\[.4em]
		\Delta(z_n^p)& =z_n^p\otimes 1 +\sum_{\substack{i+j+2k=n\\0\le j\le 1}}  (-1)^{j} F_{n,i}^p \,\,  x_1^{p j}\xb_{21}^{p k} \alpha_1^{p i} \alpha_2^p \otimes z_i^p,\\		
		\Delta(w_n^p)&=T_2^{p n}T_1^p \ot w_n^p +\sum_{\substack{i+j+2k=n\\ j\in\{0,1\}}} q^{-p(n-i)} F_{n,i}^p\,\,  T_2^{p(n-i)} w_i^p\ot \ub_{21}^{p k} \ub_1^{p j}.
	\end{align*}	
\end{coro}
\pf 
The first two identities follow from
Lemma~\ref{lemma:comultiplication-powers-Dtilde}. To prove the third
one, for each $i\in\mathbb I_{0,n}$ write $n-i=2k_i+j_i$, $k_i\in\mathbb N_0,\, j_i\in\{0,1\}$ and set
\[b_{-1}:=z_n\otimes1,\qquad b_i:=(-1)^{j_i}F_{n,i}\, x_1^{j_i}\xb_{21}^{k_i} \alpha_1^i\alpha_2\otimes z_i.\]
The elements $b_{-1},b_0,\ldots,b_n$ commute pairwise. Hence,
\[\Delta(z_n^p)=\Delta(z_n)^p=\left(\sum_{i=-1}^n b_i\right)^p=\sum_{i=-1}^n b_i^p,\]
which gives the stated formula. The formula for $\Delta(w_n^p)$ follows analogously.
\vspace{.5cm}
\epf

\subsection{Case $p=2$}
The formulas in Lemmas~\ref{lemma:relations1}, \ref{lemma:relations2}, \ref{lemma:relations3}, \ref{lemma:comultiplication-zn-wn-Dtilde} and
\ref{lemma:comultiplication-powers-Dtilde} were stated under the assumption that $p\neq2$, since some of them involve division by $2$. We explain how some of these identities can be reformulated in characteristic $2$. Since in this case we only need powers of degree at most $4$, we do not seek general formulas for arbitrary powers; instead, the remaining identities will be computed separately.
Nevertheless, many of these identities admit reformulations valid
also in characteristic $2$. In what follows, we interpret 
\[
\binom{2s+2l+1}{2}:=(s+l)(2s+2l+1),
\]
which agrees with the usual expression when $p\neq2$. If $p\neq2$ and $b=2k+j$, where
$k\in\mathbb N_0$ and $j\in\{0,1\}$, then
\begin{align*}
	\frac{[2s]^{[b]}}{2^b}x_1^b
	&=(-1)^k(s+k)^j
	\left[\prod_{l=0}^{k-1}
	\binom{2s+2l+1}{2}\right]
	\xb_1^j\xb_{21}^k,
	\\
	\frac{[2s]^{[b]}}{2^b}u_1^b
	&=(-1)^k(s+k)^j
	\left[\prod_{l=0}^{k-1}
	\binom{2s+2l+1}{2}\right]
	\ub_{21}^k\ub_1^j.
\end{align*}
Thus, expressions of the form appearing on the left-hand side can be
replaced by the corresponding expressions on the right-hand side,
which remain meaningful when $p=2$. For example,
Lemma~\ref{lemma:relations1} yields
\begin{align}\label{eq: x_{2:s}^t}
	\xb_{\{2:s\}}^t
	&=
	\sum_{\substack{i,k\geq0,\;j\in\{0,1\}\\i+j+2k=t}}
	\binom{t}{i}(-1)^k(s+k)^j
	\left[\prod_{l=0}^{k-1}
	\binom{2s+2l+1}{2}\right]
	x_2^ix_1^j\xb_{21}^k.
\end{align}
Under the convention above, this identity is valid in $H_p(q,a)$
for every $p$ and every $s\in\Bbbk$. Furthermore, suppose that $t$ is even. Formally setting
$s=\frac12$ in \eqref{eq: x_{2:s}^t} gives coefficients of the form
\[
\binom{t}{i}\left(\frac12+k\right)^j.
\]
The only potentially problematic case in characteristic $2$ is
$j=1$. In this case, the relation $i+j+2k=t$ implies that $i$ is
odd. Since $t$ is even, $\binom{t}{i}$ is therefore even, and hence
\begin{equation}\label{eq:halfbinomial}
	\binom{t}{i}\left(\frac12+k\right)
	=
	\frac12\binom{t}{i}
	+k\binom{t}{i}
	\in\mathbb N_0.
\end{equation}
Thus, when $t$ is even, the expression $x_{\{2:1/2\}}^t$ can be interpreted in characteristic $2$ by means of \eqref{eq: x_{2:s}^t} and the convention
\eqref{eq:halfbinomial}. Moreover, after formally setting
$s=\frac12$, we have
\[
\binom{2s+2l+1}{2}=\binom{2l+2}{2}=(l+1)(2l+1).
\]
Consequently, in characteristic $2$, $\prod_{l=0}^{k-1}\binom{2l+2}{2}=0$ whenever $k>1$.
It follows that, for $r\geq2$,
\begin{equation}
	x_{\{2:1/2\}}^{2r}
	=
	x_2^{2r-1}(x_2+rx_1)
	+r x_2^{2r-3}
	(x_2+(r-1)x_1)\xb_{21}.
\end{equation}

The same argument shows that the expressions $u_{\{2:s\}}^t$, $w_n^rz_m$ and $w_nz_m^t$ admit well-defined interpretations in characteristic $2$. Other formulas involve terms
of the form
\[
\binom{i+j}{j}x_{\{2:k+i/2\}}^j
\qquad\text{or}\qquad
\binom{i+j}{j}u_{\{2:k+i/2\}}^j.
\]
These expressions are problematic in characteristic $2$
only when both $i$ and $j$ are odd. In this case,
$C:=\binom{i+j}{j}$ is even. Writing $j=2j_0+1$, we obtain

\begin{align*}
	C x_{\{2:k+i/2\}}^j	&=	\left(	Cx_2+\frac{C}{2}(2k+i)x_1	\right)	x_{\{2:k+i/2\}}^{2j_0}.
\end{align*}
Since $C/2\in\mathbb N_0$, this expression is well defined in every
characteristic. The same argument applies to
$C u_{\{2:k+i/2\}}^j$. In particular, the formulas for
$w_nx_2^2$, $u_2^tz_n$ and $u_ix_2^t$ admit well-defined
interpretations in characteristic $2$.

It remains to establish the following identities. For $r\geq1$ and $j\in\{0,1\}$, we have:
\begin{align*}
	\xb_1\xb_2^{2r+j}	&=	\xb_2^{2r+j-2}(\xb_2^2+r\xb_{21})\xb_1	+j\xb_2^{2r-2}(\xb_2^2+r\xb_{21})\xb_{21},	\\
	\ub_2^{2r+j}\ub_1	&=	\ub_1(\ub_2^2+r\ub_{21})\ub_2^{2r+j-2}	+j\ub_{21}(\ub_2^2+r\ub_{21})\ub_2^{2r-2}.
\end{align*}

Finally, consider the comultiplication formulas. A parity argument analogous to that used in \eqref{eq:halfbinomial} shows that the coefficients
$F_{n,i}$ admit well-defined reductions in characteristic $2$. Consequently, Lemma~\ref{lemma:comultiplication-zn-wn-Dtilde}
remains valid when $p=2$. Together with the preceding interpretations, this also yields Lemma~\ref{lemma:comultiplication-powers-Dtilde}.

\end{document}